\documentclass[12pt]{amsart}
\usepackage[dvips]{graphicx,color}
\usepackage[dvips]{epsfig}
\usepackage{amsmath,amssymb,latexsym,amscd,xypic}

\DeclareRobustCommand{\SkipTocEntry}[4]{}
\makeatletter
\newcommand\@dotsep{4.5}
\def\@tocline#1#2#3#4#5#6#7{\relax
  \ifnum #1>\c@tocdepth 
  \else
    \par \addpenalty\@secpenalty\addvspace{#2}%
    \begingroup \hyphenpenalty\@M
    \@ifempty{#4}{%
      \@tempdima\csname r@tocindent\number#1\endcsname\relax
    }{%
      \@tempdima#4\relax
    }%
    \parindent\z@ \leftskip#3\relax \advance\leftskip\@tempdima\relax
    \rightskip\@pnumwidth plus1em \parfillskip-\@pnumwidth
    #5\leavevmode\hskip-\@tempdima #6\relax
    \leaders\hbox{$\m@th
      \mkern \@dotsep mu\hbox{.}\mkern \@dotsep mu$}\hfill
    \hbox to\@pnumwidth{\@tocpagenum{#7}}\par
    \nobreak
    \endgroup
  \fi}
\makeatother 

\DeclareFontFamily{OT1}{rsfs}{}
\DeclareFontShape{OT1}{rsfs}{n}{it}{<-> rsfs10}{}
\DeclareMathAlphabet{\curly}{OT1}{rsfs}{n}{it}

\def\scup{\mathbin{\text{\scriptsize$\cup$}}}

\newcommand\M{\overline M}

\newcommand\proj{\mathbb P}

\newcommand\Q{\mathbb Q}

\newcommand\len{\ell}

\newcommand\com{\mathbb C}

\newcommand{\Sn}{\Sigma_n}
\newcommand{\SL}{\mathsf{SL}}

\newcommand{\rarr}{\rightarrow}

\newcommand\Into{\ar@{^{ (}->}[r]}

\newcommand\Aut{\operatorname{Aut}}

\newcommand\beq{\begin{equation}}
\newcommand\eeq{\end{equation}}

\newtheorem{thm}{Theorem}

\newtheorem{cor}[thm]{Corollary}

\newtheorem{prop}{Proposition}

\title{\textbf{Tautological and non-tautological 
cohomology of the moduli space of curves}}
\author{C. Faber and R. Pandharipande}
\date{January 2011}

\begin{document}

\begin{abstract} \noindent
After a short exposition of the basic properties
of the tautological ring of $\overline{M}_{g,n}$,
we explain three methods of detecting non-tautological
classes in cohomology.
The first is via curve counting over finite
fields.
The second is by obtaining length bounds on
the action of the symmetric
group $\Sn$ on tautological classes.
The third is via classical boundary
geometry. 
Several new non-tautological classes are found.
\end{abstract}

\maketitle

\setcounter{tocdepth}{1} 
\tableofcontents

\setcounter{section}{-1}
\section{Introduction}
\subsection{Overview}
Let ${\overline{M}_{g,n}}$ be the moduli space of 
Deligne-Mumford stable curves of
genus $g$ with $n$ marked points.
The cohomology of $\overline{M}_{g,n}$ has a distinguished subring of 
tautological classes 
$$RH^*(\overline{M}_{g,n}) \subset H^*(\overline{M}_{g,n}, \mathbb{Q})\ $$
studied extensively since Mumford's seminal article \cite{M}.
While effective methods for exploring the
tautological ring have been developed over the years,
the structure of the non-tautological
classes remains mysterious.
Our goal here, after 
reviewing the basic definitions and properties of
$RH^*(\overline{M}_{g,n})$ in Section \ref{ooo}, is to
present three approaches for detecting and studying non-tautological
classes.

\subsection{Point counting and modular forms}
Since the moduli space of  Deligne-Mumford stable curves is
defined over $\mathbb{Z}$, reduction to finite fields
$\mathbb{F}_q$ is well-defined.
  Let $\overline{M}_{g,n}(\mathbb{F}_q)$ 
denote the set of $\mathbb{F}_q$-points.
For various ranges of $g$, $n$, and $q$, counting the number
of points of $\overline{M}_{g,n}(\mathbb{F}_q)$
is feasible.
A wealth of information about $H^*(\overline{M}_{g,n}, \mathbb{Q})$
can be then obtained from the Lefschetz fixed point
formula applied to Frobenius.

The first examples where point
counting reveals non-tautological cohomology occur in genus 1.
The relationship between point counting
and elliptic modular forms is discussed in Section \ref{ttt}.
By interpreting the counting results in genus 2, 
a conjectural description
of the {\em entire}
cohomology of $\overline{M}_{2,n}$ has been found
by Faber and van der Geer \cite{FvdG} in terms of
Siegel modular forms.
The formula is consistent with
point counting data for  $n\leq 25$. In
fact, large parts of the formula have been proven.
The genus 2 results are presented in Section \ref{tttt3}.

In genus 3, a more complicated investigation involving
Teichm\"uller modular forms is required. The situation is
briefly summarized in Section \ref{gen333}. As the genus
increases, the connection between point counting
and modular forms becomes more difficult to understand.

\subsection{Representation theory}
The symmetric group $\Sn$ acts on $\overline{M}_{g,n}$
by permuting the markings.
As a result, a canonical $\Sn$-representation
on $H^*(\overline{M}_{g,n},\mathbb{Q})$ is obtained.

Studying the $\Sn$-action on $H^*(\overline{M}_{g,n},\mathbb{Q})$
provides a second approach
to the non-tautolog\-ical cohomology. In Section \ref{rt},
we establish an upper bound for the 
length{\footnote{The number of parts in the corresponding partition.}} of
the irreducible $\Sn$-representations occurring in the tautological ring
$R^*(\overline{M}_{g,n})$.
In many cases, the bound is sharp. Assuming the conjectural formulas
for $H^*(\overline{M}_{2,n},\mathbb{Q})$ 
obtained by point counting,
we find several classes of Hodge type
which are presumably algebraic
(by the Hodge conjecture)
but cannot possibly be tautological, because the length of
the corresponding $\Sigma_n$-representations is too large.
The first occurs in $\overline{M}_{2,21}$.

The proofs of the length bounds for the $\Sn$-action on $R^*(\overline{M}_{g,n})$
are obtained by studying  representations induced
from the boundary strata. The strong vanishing of Proposition 2 of
\cite{FPrel}
with tautological boundary terms plays a crucial role
in the argument.

\subsection{Boundary geometry}
The existence of non-tautological cohomology classes of
Hodge type was earlier
established by Graber and Pandharipande \cite{GP2}. In particular, 
 an explicit such class in the boundary of
 $\overline{M}_{2,22}$ was found.
In Section \ref{bg}, we revisit
the non-tautological boundary constructions. The 
old examples are obtained in simpler ways and new
examples in  $\overline{M}_{2,21}$ are found.
The method is by straightforward intersection theory
on the moduli space of curves. 
Finally, we connect the boundary constructions in
$\overline{M}_{2,21}$ to the representation investigation
of Section \ref{rt}.

\subsection{Acknowledgements}
We thank J. Bergstr\"om, C. Consani, G.~van der Geer, T. Graber, and A. Pixton
for many discussions related to the cohomology of the moduli
space of curves. The gentle encouragement of the
editors G. Farkas and I. Morrison played a helpful role.

C.F. was supported by
the G\"oran Gustafsson foundation for research in natural sciences
and medicine and grant 622-2003-1123 from the Swedish Research Council.
R.P. was partially supported by NSF grants DMS-0500187
and DMS-1001154.
The paper was 
completed 
while R.P. was visiting the 
Instituto Superior T\'ecnico in Lisbon with
 support from  a Marie Curie fellowship and
a grant from the Gulbenkian foundation.

\section{Tautological classes} \label{ooo}
\subsection{Definitions}
\label{defff}
Let ${\overline{M}_{g,n}}$ be the moduli space of stable curves of
genus $g$ with $n$ marked points
defined over ${\mathbb C}$.
Let $A^*(\overline{M}_{g,n},\mathbb{Q})$ denote the Chow ring.
The system of
tautological rings is defined{\footnote{We follow here
the definition of tautological classes in Chow given in \cite{FPrel}.
Tautological classes in cohomology are discussed in Section \ref{coh}.
}}
to be the set of smallest $\Q$-subalgebras of the Chow rings,
$$R^*(\overline{M}_{g,n}) \subset A^*(\overline{M}_{g,n},\mathbb{Q}),$$
satisfying the
following two properties:
\begin{enumerate}
\item[(i)] The system is closed under push-forward via
all maps forgetting markings:
$$\pi_*: R^*(\overline{M}_{g,n}) \rarr R^*(\overline{M}_{g,n-1}).$$
\item[(ii)] The system is closed under push-forward via
all gluing maps:
$$\iota_*: R^*(\overline{M}_{g_1,n_1\scup\{\star\}})
\otimes_{\Q}
R^*(\overline{M}_{g_2,n_2\scup\{\bullet\}}) \rarr
R^*(\overline{M}_{g_1+g_2, n_1+n_2}),$$
$$\iota_*: R^*(\overline{M}_{g, n\scup\{\star,\bullet\}}) \rarr
R^*(\overline{M}_{g+1, n}),$$
with attachments along the markings $\star$ and $\bullet$.
\end{enumerate}
While the definition appears restrictive,
natural algebraic constructions typically yield Chow classes
lying in the tautological ring.

\subsection{Basic examples}
Consider first 
the cotangent line classes.
For each marking $i$, let
$$\mathbb{L}_i\rightarrow \overline{M}_{g,n}$$ 
denote the associated cotangent line bundle.
By definition
$$\psi_i = c_1(\mathbb{L}_i) \in A^1(\overline{M}_{g,n}).$$
Let $\pi$ denote  the map forgetting the last marking,
\begin{equation}\label{bob}
\pi: \overline{M}_{g,n+1} \rarr \overline{M}_{g,n},
\end{equation}
and
let $\iota$ denote the gluing map,
$$\iota: \overline{M}_{g,\{1,2,\ldots,i-1,\star,i+1,\ldots, n\}} \times
\overline{M}_{0,\{\bullet,i,n+1\}} \longrightarrow \overline{M}_{g,n+1}.$$
The ${\mathbb Q}$-multiples of
the fundamental classes $[\overline{M}_{g,n}]$  
are
contained in the tautological rings (as ${\mathbb Q}$-multiples of the 
units in the subalgebras).
A direct calculation shows:
$$-\pi_*\Big( \big(\iota_*([\overline{M}_{g,n}] \times 
[\overline{M}_{0,3}])\big)^2\Big) = \psi_i.$$
Hence, the cotangent line classes lie in the tautological rings,
$$\psi_i \in R^1(\overline{M}_{g,n})\ .$$

Consider next the $\kappa$ classes  
defined via push-forward by the forgetful map \eqref{bob}, 
$$\kappa_j= \pi_*(\psi_{n+1}^{j+1}) \in A^j(\overline{M}_{g,n})\ .$$
Since  $\psi_{n+1}$ has already been shown to be tautological,
the $\kappa$ classes are
tautological by property (i),
$$\kappa_j \in R^j(\overline{M}_{g,n})\ .$$
The $\psi$ and $\kappa$ classes are
very closely related.

For a nodal curve $C$, let $\omega_C$ denote the dualizing sheaf.
The Hodge bundle $\mathbb{E}$ over $\overline{M}_{g,n}$ 
is the rank $g$ vector bundle with fiber
$H^0(C,\omega_C)$ over the moduli point
$$[C,p_1,\ldots,p_n] \in \overline{M}_{g,n} \ .$$
The $\lambda$ classes are defined by
$$\lambda_k= c_k(\mathbb{E}) \in A^k(\overline{M}_{g,n}) \ .$$
The Chern characters of $\mathbb{E}$ lie in
the tautological ring by 
 Mumford's 
Grothendieck-Riemann-Roch computation \cite{M}.
Since the $\lambda$ classes are polynomials in the
Chern characters,
$$\lambda_k \in R^k(\overline{M}_{g,n})\ .$$

The $\psi$ and $\lambda$
classes are basic elements of the tautological
ring $R^*(\M_{g,n})$ arising very often in geometric
constructions and calculations. 
The $\psi$ integrals,
\begin{equation}\label{fffr}
\int_{\M_{g,n}} \psi_1^{a_1} \cdots \psi_{n}^{a_{n}}, 
\end{equation}
can be 
evaluated by KdV constraints due to Witten and Kontsevich
\cite{Kon,W}.
Other geometric approaches to the $\psi$ integrals can be found in
\cite{Mir,OP}.
In genus 0, the simple formula
$$\int_{\M_{0,n}} \psi_1^{a_1} \cdots \psi_{n}^{a_{n}} = \binom{n-3}{a_1,\ldots,a_n}$$
holds. Fractions occur for the first time in genus 1,
$$\int_{\M_{1,1}} \psi_1 = \frac{1}{24} \ .$$
Hodge integrals,
when the $\lambda$ classes are included,
$$\int_{\M_{g,n}} \psi_1^{a_1} \cdots \psi_{n}^{a_{n}} \lambda_1^{b_1}
\ldots \lambda_g^{b_g}, $$
can be reduced to $\psi$ integrals \eqref{fffr} by Mumford's 
Grothendieck-Riemann-Roch calculation. An example
of a Hodge integral evaluation{\footnote{Here, $B_{2g}$
denotes the Bernoulli number.}} is
\begin{equation}\label{gthh3}
\int_{\overline{M}_g} \lambda_{g-1}^3 = \frac{|B_{2g}|}{2g}\frac{|B_{2g-2}|}{2g-2}
\frac{1}{(2g-2)!}
\end{equation}
The proof of \eqref{gthh3} and
several other exact formulas for Hodge integrals can be
found in \cite{FP}.

\label{intt}

\subsection{Strata}
The boundary
strata of the moduli spaces of curves correspond
to {\em stable graphs} $$A=(V, H,L, g:V \rarr {\mathbb Z}_{\geq 0}, a:H\rarr V, i: H\rarr H)$$
satisfying the following properties:
\begin{enumerate}
\item[(1)] $V$ is a vertex set with a genus function $g$,
\item[(2)] $H$ is a half-edge set equipped with a
vertex assignment $a$ and fixed point free involution
$i$,
\item[(3)] $E$, the edge set, is defined by the
orbits of $i$ in $H$, 
\item[(4)] $(V,E)$ define a {\em connected} graph,
\item[(5)] $L$ is a set of numbered legs attached to the vertices,
\item[(6)] For each vertex $v$, the stability condition holds:
$$2g(v)-2+ n(v) >0,$$
where $n(v)$ is the valence of $A$ at $v$ including both half-edges and legs.
\end{enumerate}
The data of the topological type of a generic curve in a boundary
stratum of the moduli space is recorded in the stable graph.

Let $A$ be a stable graph. The genus of $A$ is defined by
$$g= \sum_{v\in V} g(v) + h^1(A).$$
Define the moduli space
$\M_A$ by the product
$$\M_A =\prod_{v\in V(A)} \M_{g(v),n(v)}.$$
There is a
canonical
morphism{\footnote{
To construct $\xi_A$,
a family of stable pointed curves over $\M_A$ is required.  Such a family
is easily defined
by attaching the pull-backs of the universal families over each of the
$\M_{g(v),n(v)}$  along the sections corresponding to half-edges.}} 
$$\xi_{A}: \M_{A} \rarr \M_{g,n}$$ 
with image equal to the boundary stratum
associated to the graph $A$.  
By repeated use of property (ii),
$$\xi_{A*}[\M_A] \in R^*(\M_{g,n}) \ .$$

We can now describe a set of additive generators for
$R^*(\overline{M}_{g,n})$.
Let~$A$ be a stable graph of genus $g$ with $n$ legs.
For each vertex
$v$ of $A$, let
$$\theta_v\in R^*(\overline{M}_{g(v), n(v)})$$ be an
arbitrary monomial in the
$\psi$ and $\kappa$ classes of the vertex moduli space.
The following result is proven in \cite{GP2}.

\begin{thm}\label{maude}
 $R^*(\overline{M}_{g,n})$ is generated additively by classes of the form
$$\xi_{A*}\Big(\prod_{v\in V(A)} \theta_v\Big).$$
\end{thm}

By the dimension grading, the list of generators provided
by Theorem \ref{maude} is finite. 
Hence, we obtain the following result.

\begin{cor} \label{fdr}
We have
$\dim_{\mathbb{Q}} \, R^*(\overline{M}_{g,n}) < \infty$.
\end{cor}

\subsection{Further properties}
\label{furth}
The following two formal properties of the full system of
tautological rings 
$$R^*(\overline{M}_{g,n}) \subset A^*(\overline{M}_{g,n},\mathbb{Q}),$$
are a consequence of properties (i) and (ii):

\begin{enumerate}
\item[(iii)] The system is closed under pull-back via 
the forgetting and gluing maps.

\item[(iv)] $R^*(\overline{M}_{g,n})$ is an ${\mathbb{S}}_n$-module
via the permutation action on the markings.
\end{enumerate}

Property (iii) follows from the well-known boundary
geometry of the moduli space of curves. A
careful treatment using the additive basis of Theorem \ref{maude}
can be found in \cite{GP2}. 
The meaning of property (iii) for the reducible
gluing is that the K\"unneth components of the pull-backs
of tautological classes are tautological.
Since 
the defining properties (i) and (ii) are
symmetric with respect to the marked points,
property (iv) holds.

\subsection{Pairing}
Intersection theory on the moduli space $\overline{M}_{g,n}$
yields a canonical pairing
$$\mu: R^{k}(\M_{g,n}) \times R^{3g-3+n-k}(\M_{g,n}) \rightarrow \mathbb{Q}$$
defined by
$$\mu(\alpha,\beta) = \int_{\M_{g,n}} \alpha \scup \beta \ .$$
While the pairing $\mu$ has been speculated to be perfect in 
\cite{FPlog,HL,P}, very few results are known.

The pairing $\mu$ can be effectively computed on the
generators of Theorem \ref{maude} by the following
method.
The pull-back property (iii) may be used repeatedly to
reduce the calculation of $\mu$ on
the generators to integrals of the form
$$\int_{\M_{h,m}} \psi_1^{a_1} \cdots \psi_{m}^{a_m} \cdot
\kappa_1^{b_1} \cdots \kappa_r^{b_r} \ . $$
By a well-known calculus, the $\kappa$ classes can be removed
to yield a sum of purely $\psi$ integrals
$$\int_{\M_{h,m+r}} \psi_1^{a_1} \cdots \psi_{m+r}^{a_{m+r}}, $$
see \cite{AC}. As discussed in Section \ref{intt}, the
$\psi$ integrals can be evaluated by KdV constraints.

\subsection{Further examples}
We present here two geometric constructions which also
yield tautological classes.
The first is via stable maps  and the
second via moduli spaces of Hurwitz covers.

Let $X$ be a nonsingular projective variety, and let
$\overline{M}_{g,n}(X,\beta)$ be the moduli space of stable
maps{\footnote{We refer the reader to \cite{FulP}
for an introduction to the subject. A discussion of
obstruction theories and virtual classes can be found in \cite{B,BF,LiT}.}} 
representing $\beta\in H_2(X,{\mathbb Z})$.
Let $\rho$ denote the map to the moduli of curves,
$$\rho: \overline{M}_{g,n}(X,\beta) \rarr {\overline M}_{g,n}.$$
The moduli space $\overline{M}_{g,n}(X,\beta)$
carries a virtual class 
$$[\overline{M}_{g,n}(X,\beta)]^{vir} \in A_*(\overline{M}_{g,n}(X,\beta))$$
obtained from the canonical
obstruction theory of maps. 

\begin{thm} \label{tttt}
Let $X$ be a nonsingular projective toric
variety. Then,
$$\rho_* [\overline{M}_{g,n}(X,\beta)]^{vir} \in 
R^*(\overline{M}_{g,n}).$$
\end{thm}
The proof follows directly from the virtual
localization formula of \cite{GP}.
If $[C]\in M_g$ is a general moduli point, then
 $$\rho_* [\overline{M}_{g}(C,1)]^{vir}= [C] \in 
A^{3g-3}(\overline{M}_{g}).$$
Since not all of $A^{3g-3}(\overline{M}_g)$ is expected to be
tautological,
Theorem \ref{tttt} is unlikely to hold for $C$. However, 
the result perhaps holds for a much more general class
of varieties $X$. 
In fact,
we do not know{\footnote{
We do not know
any examples at all of nonsingular projective $X$ where
$$\rho_* [\overline{M}_{g,n}(X,\beta)]^{vir} \notin 
RH^*(\overline{M}_{g,n}).$$
See \cite{LP} for a discussion.}}
any nonsingular projective variety
defined over $\bar{\mathbb Q}$ for which
Theorem \ref{tttt} is expected to be false.

The moduli spaces of Hurwitz covers of $\proj^1$
also define natural classes on the moduli space of curves. 
Let $\mu^1, \ldots, \mu^m$ be $m$ partitions of equal size $d$ satisfying
$$2g-2+2d = \sum_{i=1}^m \Big( d- \len(\mu^i) \Big),$$
where $\ell(\mu^i)$ denotes the length of the partition $\mu^i$.
The moduli space of Hurwitz covers, $$H_{g}(\mu^1,\ldots,\mu^m)$$
parameterizes morphisms,
$$f: C \rightarrow  \mathbb{P}^1,$$
where $C$ is
a complete, connected, nonsingular curve 
with marked profiles $\mu^1, \ldots, \mu^m$ over $m$ ordered points of the
target (and no ramifications elsewhere).  
The moduli space of Hurwitz covers is a dense open set of 
the compact 
moduli space of admissible covers \cite{HM},
$$H_{g}(\mu^1,\ldots,\mu^m) \subset  
\overline{H}_{g}(\mu^1,\ldots,\mu^m).$$
Let $\rho$ denote the map to the moduli of curves,
$$\rho: \overline{H}_{g}(\mu^1,\ldots,\mu^m) \rightarrow {\overline M}_{g, 
\sum_{i=1}^m \len(\mu^i)}.$$
The following is a central result of \cite{FPrel}.

\begin{thm} \label{hht}
The push-forwards of the fundamental classes lie in
the tautological ring,
$$\rho_*[\overline{H}_{g}(\mu^1,\ldots,\mu^m)] \in 
R^*(\overline{M}_{g,\sum_{i=1}^m \len(\mu^i)})\ . $$
\end{thm}

The admissible covers in Theorem \ref{hht} are of
$\proj^1$. 
The moduli spaces of admissible covers of higher
genus targets can lead to non-tautological classes \cite{GP2}.

\subsection{Tautological rings} \label{coh}
The tautological subrings 
\begin{equation}\label{fbm}
RH^*(\M_{g,n}) \subset H^*(\M_{g,n},\mathbb{Q})
\end{equation}
are defined to be the images of 
$R^*(\M_{g,n})$ under the cycle class map
$$A^*(\M_{g,n},\mathbb{Q}) \rightarrow H^*(\M_{g,n},\mathbb{Q})\ .$$
The tautological rings $RH^*(\overline{M}_{g,n})$  
could alternatively be defined
as the smallest system of subalgebras \eqref{fbm} 
closed under push-forward by all forgetting and
gluing maps (properties (i) and (ii)). Properties
(iii) and (iv) also hold for $RH^*(\M_{g,n})$.

Tautological rings in Chow and cohomology
may be defined for all stages of the standard 
moduli filtration:
\begin{equation}
\label{fff}
\overline{M}_{g,n} \supset M^c_{g,n} \supset  M^{rt}_{g,n} .
\end{equation}
Here,
$M_{g,n}^c$ is the moduli space of 
stable curves of compact type (curves
with tree dual graphs, or equivalently, with
compact Jacobians), and
$M^{rt}_{g,n}$ is the moduli space of
stable curves with rational tails (which lies over the moduli space
of nonsingular curves $M_g \subset \M_g$).
The tautological rings
$$R^*(M^c_{g,n}) \subset A^*(M^c_{g,n},\mathbb{Q}),\ \ \  
RH^*(M^c_{g,n}) \subset H^*(M^c_{g,n},\mathbb{Q}), $$ 
$$R^*(M^{rt}_{g,n}) \subset A^*(M^{rt}_{g,n},\mathbb{Q}),\ \ \   
RH^*(M^{rt}_{g,n}) \subset H^*(M^{rt}_{g,n},\mathbb{Q})\  $$
are all defined as the images of $R^*(\M_{g,n})$
under the natural restriction and cycle class maps.

Of all the tautological rings, the most studied
case by far is $R^*(M_g)$. A complete (conjectural) structure
of $R^*(M_g)$ is proposed in \cite{F} with important
advances made in \cite{I,L,Mo}.
Conjectures
for the compact type cases $R^*(M_{g,n}^c)$ can be found
in \cite{FPlog}. See \cite{FPlam,P1,P2} for positive results for
compact type. Tavakol \cite{Tav1,Tav2}
has recently proved that $R^*(M_{1,n}^c)$ and $R^*(M^{rt}_{2,n})$
are Gorenstein, with socles in degrees $n-1$ and $n$, respectively.

\subsection{Hyperelliptic curves}
As an example, we calculate the class of the hyperelliptic locus
$$H_g \subset M_g$$
following \cite{M}.

Over the moduli point $[C,p]\in M_{g,1}$ there is a canonical map
$$\phi: H^0(C,\omega_C)  \rightarrow H^0(C, \omega/\omega(-2p))$$
defined by evaluating sections of $\omega_C$ at 
the marking $p\in C$.
If we define $\mathbb{J}$ to be the rank 2 bundle over $M_{g,1}$ with
fiber $H^0(C, \omega_C/\omega_C(-2p))$ over $[C,p]$, then we obtain
a morphism
$$\phi: \mathbb{E}  \rightarrow \mathbb{J}$$
over $M_{g,1}$.
By classical curve theory,
the map $\phi$ fails to be surjective precisely when $C$ is
hyperelliptic {and} $p\in C$ is a Weierstrass point.
Let 
$$\Delta \subset M_{g,1}$$ be
the degeneracy locus of $\phi$ of pure codimension
$g-1$.
By the Thom-Porteous formula \cite{Ful},
$$[\Delta] =  \left( \frac{c(\mathbb{E}^*)}{c(\mathbb{J}^*)}\right)_{g-1} \ . $$
By using the jet bundle sequence
$$0 \rightarrow \mathbb{L}_1^2 \rightarrow \mathbb{J} \rightarrow 
\mathbb{L}_1 \rightarrow 0,$$ we conclude
$$[\Delta] =  \left( \frac{1-\lambda_1 +\lambda_2 - \lambda_3 +\ldots +(-1)^g
\lambda_g}{(1-\psi_1) (1-2\psi_1)}\right)_{g-1} \in R^{g-1}(M_{g,1})\ .$$

To find a formula for the hyperelliptic locus,
 we view
$$\pi: M_{g,1} \rightarrow M_g$$
as the universal curve over moduli space.
Since each hyperelliptic curve has $2g+2$ Weierstrass points,
$$[H_g] = \frac{1}{2g+2} \pi_*([\Delta]) \in R^{g-2}(M_g) \ .$$
By Theorem \ref{hht}, the class of the closure of the hyperelliptic
locus is also tautological,
$$[\overline{H}_g] \in R^{g-2}(\overline{M}_g), $$
but no simple formula is known to us.

\section{Point counting and elliptic modular forms} \label{ttt}

\subsection{Elliptic modular forms}
We present here an introduction to the close relationship
between modular forms and the cohomology of moduli spaces.
The connection is most direct between elliptic modular forms and
moduli spaces
of elliptic curves.

Classically, a modular form is a holomorphic function $f$
on the upper half plane
$$\mathbb{H}=\{z\in\mathbb{C}:\text{Im}\,z>0\}$$
with an amazing amount of symmetry.
Precisely, a modular form of {\em weight} $k\in \mathbb{Z}$ satisfies  
the functional equation 
$$f\bigg(\frac{az+b}{cz+d}\bigg)=(cz+d)^k f(z)$$
for all $z\in\mathbb{H}$ and all
$\left(\begin{smallmatrix}a&b\\ c&d\\ \end{smallmatrix}\right)\in 
\SL(2,\mathbb{Z}).$
By
the symmetry for 
$\left(\begin{smallmatrix}-1&0\\ 0&-1\\ \end{smallmatrix}\right)$,
 $$f(z)=(-1)^kf(z),$$
 so $k$ must be even.
Using 
$\left(\begin{smallmatrix}1&1\\ 0&1\\ \end{smallmatrix}\right)$, we see
$$f(z+1)=f(z),$$
so $f$ can be written as a function
of $q=\exp(2\pi iz)$.
We further require $f$ to be holomorphic at $q=0$.
The Fourier expansion is
$$f(q)=\sum_{n=0}^{\infty}a_nq^n.$$
If $a_0$ vanishes, $f$ is called a {\em cusp} form
(of weight $k$ for $\SL(2,\mathbb{Z})$).

Well-known examples of modular forms 
are the Eisenstein series{\footnote{
 Recall the Bernoulli numbers $B_k$ are given by $\frac{x}{e^x-1}=
\sum B_k\frac{x^k}{k!}$ and $\sigma_m(n)$ is the sum of the
$m$-th powers of the positive divisors of $n$.}}
$$E_k(q)=1-\frac{2k}{B_k}\sum_{n=1}^{\infty}\sigma_{k-1}(n)q^n$$
of weight $k\geq 4$
and the discriminant cusp form
\begin{equation}\label{disc}
\Delta(q)=\sum_{n=1}^{\infty}\tau(n)q^n=q\prod_{n=1}^{\infty}(1-q^n)^{24}
\end{equation}
of weight $12$ with the Jacobi product expansion (defining the
Ramanujan $\tau$-function).
In both of the above formulas, the right sides are the
Fourier expansions.

Let $M_k$ be the vector space of holomorphic modular
forms of weight $k$ for $\SL(2,\mathbb{Z})$. 
The ring of modular forms is freely generated \cite{Kob} by $E_4$ and $E_6$,
$$\bigoplus_{k=0}^\infty M_k = \mathbb{C}[E_4,E_6]\ ,$$
and the ideal $\bigoplus_kS_k$ of cusp forms is generated by $\Delta$.
Let $m_k=\dim\,M_k$ and $s_k=m_k-1$.

For every positive integer $n$,
there exists a naturally defined {\em Hecke operator\/} 
$$T_n: M_k\rightarrow M_k\ .$$ 
The $T_n$ commute with each other
and preserve the subspace $S_k$. Moreover, $S_k$ has
a basis of simultaneous eigenforms, which are orthogonal with respect to the
{\em Petersson inner product\/} and can be normalized to have $a_1=1$.
The $T_n$-eigenvalue of such an eigenform equals its $n$-th
 Fourier coefficient \cite{Kob}.

\subsection{Point counting in genus 1}\label{modee}
While
the moduli space $M_{1,1}$ of elliptic curves can be realized over $\mathbb{C}$
as the analytic space $\mathbb{H}/\SL(2,\mathbb{Z})$, we will
view $M_{1,1}$ here as an
algebraic variety over a field $k$ or a scheme (stack) over $\mathbb{Z}$.

Over $\overline{k}$, the $j$-invariant classifies elliptic curves, so 
$$M_{1,1}\ \cong_{\overline{k}}\ \mathbb{A}^1$$
as a coarse moduli space.
But what happens over a finite field $\mathbb{F}_p\,$, where $p$ is a prime
number? Can we calculate $\#M_{1,1}(\mathbb{F}_p)$?
If we count elliptic curves
$E$ over $\mathbb{F}_p$ (up to
$\mathbb{F}_p$-isomorphism) with weight factor
\begin{equation}\label{abcv}
\frac{1}{\#\text{Aut}_{\mathbb{F}_p}(E)}\ , 
\end{equation}
we obtain the expected result.

\begin{prop} \label{gettt}
We have
$
\#M_{1,1}(\mathbb{F}_p)=p$. 
\end{prop}

\begin{proof}
The counting is very simple for $p\neq 2$.
Given an elliptic curve with a point $z\in E$ defined over
$\mathbb{F}_p$, we obtain a degree 2 morphism
$$\phi:E \rightarrow \mathbb{P}^1$$
from the linear series associated to $\mathcal{O}_E(2z)$.
If we view the image of $z$ as $\infty\in \mathbb{P}^1$, the branched
covering $\phi$ expresses $E$ in Weierstrass form,
$$y^2 = x^3+ ax^2+bx+c \ \ \ \text{with} \ \  a,b,c\in \mathbb{F}_p,$$
where the cubic on the right has distinct roots in
$\overline{\mathbb{F}}_p$.
The number of monic cubics in $x$ is $p^3$, the number with
a single pair of double roots is $p^2-p$, and the number with
a triple root is $p$.
Hence, the number of Weierstrass forms is $p^3-p^2$.

Since only  $\infty\in \mathbb{P}^1$ is distinguished,
we must further divide by the group of affine transformations
of 
$$\mathbb{A}^1=\mathbb{P}^1\setminus \{ \infty \}$$ over $\mathbb{F}_p$. Since the order
of the affine transformation group is $p^2-p$, we conclude
$$\#M_{1,1}(\mathbb{F}_p)=\frac{p^3-p^2}{p^2-p} = p\ .$$
We leave to the reader to check the above counting weights
curves by the factor \eqref{abcv}. We also leave the $p=2$
case to the reader.
\end{proof}

Motivated by Proposition \ref{gettt}, we will study the
number of points of $M_{1,n}(\mathbb{F}_p)$ 
weighted by the order of the $\mathbb{F}_p$-automorphism group
of the marked curve.
For example, 
$$\#M_{1,6}(\mathbb{F}_2)=0,$$ 
since an elliptic curve $E$
over $\mathbb{F}_2$ contains at most $5$ distinct points by
the Weil bound \cite{HAG},
$$| \#E(\mathbb{F}_2)- (2+1)| \leq 2\sqrt 2\ .$$

To investigate the behavior of 
$M_{1,n}(\mathbb{F}_p)$, we can effectively enumerate
$n$-pointed elliptic curves over $\mathbb{F}_p$ by
computer.
Interpreting the data, we find
for $n\le10$, 
\begin{equation}\label{hmmr}
\#M_{1,n}(\mathbb{F}_p)=f_n(p),
\end{equation}
where $f_n$ is a monic polynomial of degree $n$ with integral coefficients.
The moduli spaces $M_{1,n}$ are rational varieties \cite{Pasha} for $n\leq 10$.
The polynomiality \eqref{hmmr}
has been proved directly by Bergstr\"om \cite{Bg3}
using geometric constructions (as in the proof of Proposition
\ref{gettt}). 
However for $n=11$, the computer counts
exclude the possibility of a degree 11  
polynomial for $\#M_{1,11}(\mathbb{F}_p)$. The data
suggest
\begin{equation}\label{tau4}
\#M_{1,11}(\mathbb{F}_p)=f_{11}(p)-\tau(p).
\end{equation}
For higher $n$, the $p$-th Fourier coefficients of the other
Hecke cusp eigenforms eventually appear in the counting
as well (multiplied with
polynomials in~$p$).

\subsection{Cohomology of local systems}
\label{cls}
To understand what is going on, we recall
the number of points of $M_{1,n}$ over $\mathbb{F}_p$ equals the
trace of Frobenius on the Euler characteristic of
the compactly supported $\ell$-adic cohomology of $M_{1,n}$.
The Euler characteristic has good additivity properties
and the multiplicities in applying the Lefschetz fixed point
formula to ${\rm Frob}_p$ are all equal to~$1$. 
The cohomology  of projective space yields traces which are
polynomial in $p$,
$$\#\mathbb{P}^n(\mathbb{F}_p)=p^n+p^{n-1}+\dots+p+1,$$
since the trace of Frobenius on the Galois
representation $\mathbb{Q}_{\ell}(-i)$ 
equals $p^i$. 

As is well-known, the fundamental class
of an algebraic subvariety is of Tate type (fixed
by the Galois group). The
Tate conjecture
essentially asserts the converse. 
If equation \eqref{tau4} holds for all $p$, then
$\#M_{1,11}(\mathbb{F}_p)$ will fail to fit a
degree 11 polynomial (even for all but finitely many $p$).
We can then conclude $M_{1,11}$ possesses cohomology not represented
by
algebraic
classes defined over $\mathbb{Q}$. Such cohomology,
in particular, can not be tautological.

In fact, the moduli space $M_{1,n}(\mathbb{C})$ has cohomology related to cusp forms. 
The simplest connection is the construction of a 
holomorphic differential form on $M_{1,11}$ from
the discriminant  \eqref{disc}.
We view $M_{1,11}$ as an open subset{\footnote{In fact,
$M_{1,11}$ is obtained by removing diagonal loci from the right side.}} 
$$M_{1,11} \subset \frac{\mathbb{H} \times \com^{10}}{\SL(2,\mathbb{Z}) \times
(\mathbb{Z}^2)^{10}}\ .$$
As usual, $\SL(2,\mathbb{Z})$ acts on $\mathbb{H}$ via
linear fractional transformations. At the point
$(z, \zeta_1, \ldots, \zeta_{10}) \in \mathbb{H} \times \com^{10}$,
the action is
\begin{multline*}
(\left(\begin{smallmatrix}a&b\\ c&d\\ \end{smallmatrix}\right)
, (x_1,y_1), \ldots, (x_{10},y_{10})) 
\cdot (z, \zeta_1, \ldots, \zeta_{10})
= \\ \Bigg( 
\left(\frac{az+b}{cz+d}\right), \frac{\zeta_1}{cz+d}+x_1+y_1 
\left(\frac{az+b}{cz+d}\right), 
\\ \ldots,
\frac{\zeta_{10}}{cz+d}+x_{10}+y_{10} 
\left(\frac{az+b}{cz+d}\right)
\Bigg)\ .
\end{multline*}
A direct verification using the weight 12
functional equation for the discriminant shows
the holomorphic 11-form
\begin{equation}\label{nnhh}
\Delta(e^{2\pi iz})\, dz \wedge d\zeta_1 \wedge \ldots \wedge d\zeta_{10}
\end{equation}
on $\mathbb{H} \times \com^{10}$ is invariant under the
action and descends to $M_{1,11}$. Since
the discriminant is a cusp form, the form \eqref{nnhh}
extends to a nontrivial element of $H^{11,0}(\overline{M}_{1,11},\com)$.
In fact, 
$$H^{11,0}(\overline{M}_{1,11},\com) \cong \com$$
by calculations of \cite{GetzSemi} (the $\Sigma_{11}$-representation
on $H^{11,0}(\overline{M}_{1,11},\com)$ is alternating).
As a consequence, $\overline{M}_{1,11}$
is irrational.

Returning to the open locus $M_{1,n}$, we can consider the
fibration
$$\pi:M_{1,n} \rightarrow M_{1,1}$$
with fibers given by open subsets of $E^{n-1}$ (up to automorphisms).
The only interesting cohomology of $E$ is  ${H}^1(E,\mathbb{Q})$.
The motivic Euler characteristic of the cohomology of the 
fibers of $\pi$ can be expressed in terms of the
symmetric powers of ${H}^1(E,\mathbb{Q})$.
Letting $E$ vary, the $a$-th symmetric power gives rise to a
{\em local system\/} $\mathbb{V}_a$ on $M_{1,1}$.
The cohomology of these local systems was studied in detail by Shimura.
A basic result \cite{Shi} is the {\em Shimura isomorphism}, 
$${H}^1_!(M_{1,1},\mathbb{V}_a)\otimes\mathbb{C} = S_{a+2} \oplus
\overline{S}_{a+2}\,.$$
Here, $H^i_!$ is the {\em inner cohomology\/}, the image
of the cohomology $H^i_c$ with compact support in the usual
cohomology $H^i$. The Shimura isomorphism gives a connection
between $M_{1,11}$ and the space of cusp forms $S_{12}$. 
The inner cohomology group
${H}^1_!(M_{1,1},\mathbb{V}_a)$ has a pure Hodge structure
of weight $a+1$ and $S_{a+2}$ has Hodge type $(a+1,0)$. 
We have found
non-algebraic cohomology, but we still need to understand the
contribution to the trace of Frobenius.

As shown by Deligne \cite{Del}, Hecke operators can be defined 
on the inner cohomology group ${H}^1_!(M_{1,1},\mathbb{V}_a)$
compatibly with the Shimura isomorphism and the earlier 
operators on the spaces of cusp forms.
The Eichler-Shimura congruence relation \cite{Del,Eich,Shi} 
establishes a connection
between the Hecke operator $T_p$ on
${H}^1_!(M_{1,1},\mathbb{V}_a)\otimes\mathbb{Q}_{\ell}$ and the
Frobenius at $p$. More precisely, $T_p$ equals the sum of
Frobenius and Verschiebung at $p$ (the adjoint of Frobenius with 
respect to the natural scalar product).
As explained in \cite{Del},
the results finally show the trace of Frobenius at $p$ 
on ${H}^1_!(M_{1,1},\mathbb{V}_a)\otimes\mathbb{Q}_{\ell}$
to be equal to the
trace of the Hecke operator $T_p$ on $S_{a+2}$, which is the sum
of the $p$-th Fourier coefficients of the normalized Hecke cusp eigenforms.

We are now in a position to derive the mysterious
counting formula 
$$\#M_{1,11}(\mathbb{F}_p)=f_{11}(p)-\tau(p).$$
The kernel of the map 
$$H^i_c(M_{1,1}, \mathbb{V}_a)\to H^i(M_{1,1}, \mathbb{V}_a)$$ 
is the {\em Eisenstein
cohomology}. Like the inner cohomology, 
the Eisenstein cohomology for $a>0$ is concentrated
in degree~$1$. It consists of Hodge-Tate classes of weight~$0$ which
over $\mathbb{C}$ are represented by suitably normalized Eisenstein
series. Writing $S[a+2]$ for ${H}^1_!(M_{1,1},\mathbb{V}_a)$,
we obtain 
\begin{equation}\label{ghh3}
e_c(M_{1,1},\mathbb{V}_a)=-S[a+2]-1
\end{equation}
for $a>0$ even, while the Euler characteristic vanishes for $a$ odd.
A motivic construction of
$S[a+2]$ can be found in \cite{CF,Scholl}, see also \cite{Petersen2}.
By the discussion in the previous paragraph, the trace of
Frobenius picks up $-\tau(p)$ from \eqref{ghh3}.
As another consequence of \eqref{ghh3}, we find the
following formula for the integer valued Euler
characteristic,
\begin{equation*}
E_c(M_{1,1},\mathbb{V}_a)=-2s_{a+2}-1.
\end{equation*}
In fact,
 formula \eqref{ghh3} has been used by Getzler \cite{GetzSemi}
to calculate the full Hodge decomposition of
the cohomology of $\overline{M}_{1,n}$.

\section{Point counting and Siegel modular forms} \label{tttt3} 
\subsection{Summary}
Siegel modular forms are
related to  moduli spaces of abelian varieties.
As a consequence,
a strong relationship between such modular forms and the
cohomology of moduli spaces of curves of genus 2 is obtained,
leading to a complete (conjectural) description
of the cohomology of $\overline{M}_{2,n}$. 
Perhaps surprisingly, the resulting  description in genus 3 is
not complete.
Elliptic
and Siegel modular forms do not suffice to describe the cohomology
of~$M_{3,n}$ for~$n$ large enough --- the Teich\-m\"uller
modular forms introduced by Ichikawa \cite{Ichi}
are needed.
In genus 2 and 3, data from point counting over finite fields
is used to formulate the conjectures and to explore the cohomology.

\subsection{Siegel modular forms}
In higher genus $g>1$, the analogue of elliptic modular
forms are {\em Siegel modular forms\/}. The latter are
defined on
$\mathbb{H}_g$, the space of $g\times g$ complex symmetric
matrices with positive-definite imaginary part and  satisfy the
functional equation 
$$ f((az+b)(cz+d)^{-1})=\rho(cz+d) f(z)$$
for all $z\in\mathbb{H}_g$ and all
$\left(\begin{smallmatrix}a&b\\ c&d\\ \end{smallmatrix}\right)\in 
\mathsf{Sp}(2g,\mathbb{Z})$. Here, $\rho$ is a representation
of $\mathsf{GL}(g,\mathbb{C})$.
Classically, $$\rho(M)={\text{Det}}^k(M),$$
 but other irreducible
representations of $\mathsf{GL}(g,\mathbb{C})$ are just as relevant.
The form $f$ is vector valued. Since $z$ has positive
definite imaginary part,  $cz+d$ is indeed
invertible.

Hecke operators can be defined on Siegel modular forms \cite{Andr,Arak}. The
Fourier coefficients of an eigenform here contain much more
information than just the Hecke eigenvalues:
\begin{enumerate}
\item [(i)]
the coefficients are indexed by positive semi-definite
$g\times g$ symmetric matrices with integers on the diagonal
and half-integers elsewhere,
\item[(ii)]
the coefficients themselves
are vectors.
\end{enumerate}
{\em Cusp forms\/} are Siegel modular forms in the kernel
of the Siegel $\Phi$-operator:
$$(\Phi f)(z')=\lim_{t\to\infty} f\left(\begin{smallmatrix}z'&0\\
0&it\\ \end{smallmatrix}\right),$$
for $z'\in\mathbb{H}_{g-1}$.
As before, the Hecke operators preserve the
space of cusp forms.

\subsection{Point counting in genus 2}

Since 
genus $2$ curves are double covers of $\mathbb{P}^1$, 
we can effectively count them over $\mathbb{F}_p$. As before, each isomorphism
class over $\mathbb{F}_p$ is counted with the reciprocal of the
number of $\mathbb{F}_p$-automorphisms. For small $n$, the
enumeration of $M_{2,n}(\mathbb{F}_p)$ can be done by hand as
in Proposition \ref{gettt}. For large $n$, computer counting
is needed.

Bergstr\"om \cite{Bhyp} has proven the polynomiality in $p$ of
$\#M_{2,n}(\mathbb{F}_p)$ for $n\leq 7$ by geometric
stratification of the moduli space. 
Experimentally,  we find $\#M_{2,n}(\mathbb{F}_p)$ is a
polynomial 
of degree $n+3$ in $p$ for $n\leq 9$.
In the range
 $10\le n\le 13$, the function
$\tau(p)$ 
multiplied by a polynomial appears in $\#M_{2,n}(\mathbb{F}_p)$.
More precisely, computer counting predicts
$$\#M_{2,10}(\mathbb{F}_p)= f_{13}(p) + (p-9) 
\tau(p)  \ .$$
If the above formula holds, then
 the possibility of a degree $13$
polynomial for $\#M_{2,10}(\mathbb{F}_p)$ is excluded, so 
we conclude 
$M_{2,10}$ has cohomology which is not algebraic
over $\mathbb{Q}$ (and therefore non-tautological).

Higher weight modular forms appear in the counting for 
 $n=14$. Computer counts predict 
$$\#M_{2,14}(\mathbb{F}_p)=f_{17}(p)+f_5(p)\tau(p)
-13c_{16}(p)-429c_{18}(p),$$
where $c_{16}(p)$ and $c_{18}(p)$ are the Fourier coefficients
of the normalized Hecke cusp eigenforms of weights 16 and 18.
The appearance of $c_{16}(p)$ appears is perhaps not so surprising
since $$\tau= c_{12}$$ occurs in the function
$\#M_{2,10}(\mathbb{F}_p)$. By counting
$\Sigma_{14}$-equivariantly, we find  the coefficient $13$
is the dimension of the irreducible representation corresponding
to the partition $[2\,1^{12}]$. 

More surprising is the appearance of $c_{18}(p)$. Reasoning as before in 
Section \ref{modee}, 
we would expect the motive $S[18]$ to occur in the
cohomology of $M_{2,14}$. Since the Hodge types
are $(17,0)$ and $(0,17)$,
the motive 
$S[18]$ cannot possibly come from the boundary
of the $17$ dimensional moduli space,
but must come from the interior.
We would then expect a Siegel cusp form to be responsible for the
occurrence, with Hecke eigenvalues closely related
to those of the elliptic cusp form of weight $18$. There indeed 
exists such a Siegel modular form:
the classical Siegel cusp form $\chi_{10}$ of weight $10$, the product
of the squares of the $10$ even theta characteristics, is a 
Saito-Kurokawa lift of $E_6\Delta$. In particular, the Hecke eigenvalues
are related via
$$\lambda_{\chi_{10}}(p)=\lambda_{E_6\Delta}(p)+p^8+p^9=c_{18}(p)+p^8+p^9.$$
The coefficient $429$ is the dimension
of the irreducible representation corresponding
to the partition $[2^7]$.

The function $\#M_{2,15}(\mathbb{F}_p)$ is expressible in
terms of $\tau$, $c_{16}$, and $c_{18}$ as expected. 
For $n=16$, the Hecke eigenvalues of a vector valued Siegel
cusp form appear for the first time: the unique form
of type $$\rho={\text{Sym}}^6\otimes{\text{Det}} ^8$$  constructed by Ibukiyama
(see \cite{FvdG}) arises. 
For $n\leq 25$, the
computer counts of $\#M_{2,n}(\mathbb{F}_p)$  have 
to a large extent
been successfully fit by Hecke eigenvalues of Siegel cusp forms \cite{Grundh}.

\subsection{Local systems in genus 2}
To explain the experimental results  for $\#M_{2,n}(\mathbb{F}_p)$
discussed above, 
we study the cohomology of local systems on the
moduli space $M_2$.
Via the Jacobian, there is an inclusion
$$M_2 \subset A_2$$
as an open subset in the 
moduli space of principally polarized abelian surfaces.
In fact, the geometry of local systems on $A_2$ is a more
natural object of study.

The irreducible representations $V_{a,b}$ of $\mathsf{Sp}(4,\mathbb{Z})$
are indexed by integers
$$a\ge b\ge0.$$ 
Precisely, $V_{a,b}$ is the irreducible representation of
highest weight occurring in 
$\text{Sym}^{a-b}(R)\otimes\text{Sym}^b(\wedge^2 R)$, where
$R$ is the dual of the standard representation.
Since the moduli space of abelian surfaces arises
as a quotient,
$$A_2 = \mathbb{H}_2/ \mathsf{Sp}(4,\mathbb{Z}),$$
we obtain
local systems $\mathbb{V}_{a,b}$ for $\mathsf{Sp}(4,\mathbb{Z})$ 
on $A_2$ 
associated to the representations $V_{a,b}$.
The local system $\mathbb{V}_{a,b}$ is {\em regular} if
$a>b>0$.

Faltings and Chai \cite{faltchai}
relate the cohomology of $\mathbb{V}_{a,b}$ on $A_2$
to the space $S_{a-b,b+3}$ of
Siegel cusp forms of type $\text{Sym}^{a-b}\otimes\text{Det}^{b+3}$.
To start,  $H^i_c(A_2,\mathbb{V}_{a,b})$ has
a natural mixed Hodge structure with weights at most ${a+b+i}$, and
$H^i(A_2,\mathbb{V}_{a,b})$ has a mixed Hodge structure with 
weights at least
${a+b+i}$. The inner cohomology $H^i_!(A_2,\mathbb{V}_{a,b})$
therefore has a pure Hodge structure of weight $a+b+i$. Faltings \cite{falt}
had
earlier shown that $H^i_!(A_2,\mathbb{V}_{a,b})$ is concentrated
in degree~$3$ when the local system is regular.
For the Hodge structures above, the degrees of the
Hodge filtration are contained in $$\{0,b+1,a+2,a+b+3\}.$$
 Finally,
there are natural isomorphisms
$$F^{a+b+3}H^3(A_2,\mathbb{V}_{a,b})\cong M_{a-b,b+3},$$
$$F^{a+b+3}H^3_c(A_2,\mathbb{V}_{a,b})
=F^{a+b+3}H^3_!(A_2,\mathbb{V}_{a,b})
\cong S_{a-b,b+3}$$
with the spaces of Siegel modular (respectively cusp) forms of the
type mentioned 
above.
The quotients in the Hodge filtration
are isomorphic to certain explicit coherent cohomology groups, see
\cite{TRR}.

One is inclined to expect the occurrence of motives $S[a-b,b+3]$
corresponding to Siegel cusp forms
in $H^3_!(A_2,\mathbb{V}_{a,b})$, at least in the case of a regular
weight. Each Hecke cusp eigenform should contribute a
4-dimensional piece with four 1-dimensional pieces in the
Hodge decomposition of types 
$$(a+b+3,0), \ (a+2,b+1), \ (b+1,a+2), \ (0,a+b+3).$$
By the theory of automorphic representations,
the inner cohomology may contain other terms as well,
the so-called endoscopic contributions. In the present case, the
endoscopic contributions come apparently only{\footnote{However,
our use of the term endoscopic is potentially non-standard.}} 
with Hodge types $(a+2,b+1)$ and $(b+1,a+2)$.

Based on the results from equivariant point counts, 
Faber and Van der Geer \cite{FvdG} 
obtain the following explicit conjectural
formula for the  Euler characteristic of the compactly
supported cohomology:

\vspace{10pt}
\noindent {\bf{Conjecture.}}
{\em For $a>b>0$ and $a+b$ even,
\begin{multline*}
e_c(A_2,\mathbb{V}_{a,b})=-S[a-b,b+3]-s_{a+b+4}S[a-b+2]L^{b+1}
\\
+s_{a-b+2}-s_{a+b+4}L^{b+1}-S[a+3]+S[b+2]+\tfrac12(1+(-1)^a).
\end{multline*}}
\vspace{0pt}

The terms in the second line constitute the
Euler characteristic of the Eisenstein cohomology. The second term
in the first line is the endoscopic contribution.
Just as in the genus $1$ case,
the hyperelliptic involution causes the vanishing of all cohomology
when $a+b$ is odd.
However, the above conjecture is not quite
a precise statement. While 
$L$ is the Lefschetz motive and
$S[k]$ has been discussed in Section \ref{cls},
the motives $S[a-b,b+3]$ have not been
constructed yet. Nevertheless, several precise
predictions of the conjecture  will be discussed.

First, we can specialize the conjecture to yield a
prediction for the 
integer valued Euler characteristic
$E_c(A_2,\mathbb{V}_{a,b})$ . The formula from the conjecture is
\begin{multline*}
E_c(A_2,\mathbb{V}_{a,b})=-4s_{a-b,b+3}-2s_{a+b+4}s_{a-b+2}
\\
+s_{a-b+2}-s_{a+b+4}-2s_{a+3}+2s_{b+2}+\tfrac12(1+(-1)^a).
\end{multline*}
for $a>b>0$ and $a+b$ even. The lower case $s$ denotes
the dimension of the corresponding space of cusp  forms.
The dimension formula
was proved by Grundh \cite{Grundh} for $b>1$ using
earlier work of Getzler and a formula of Tsushima for 
$s_{j,k}$,
proved for $k>4$
 (and presumably true for~$k=4$).
In fact,
combining work of Weissauer \cite{Weiss} on the inner cohomology and of 
van der Geer \cite{vdG}
on the Eisenstein cohomology, one may deduce 
the implication of the conjecture obtained by taking the realizations
as $\ell$-adic Galois representations of all terms.

A second specialization of the conjecture yields a prediction
for $\#M_{2,n}(\mathbb{F}_p)$ via the eigenvalues of Hecke
cusp forms. The prediction has been checked for $n\leq 17$ 
and $p\le23$. The prediction for the trace of $\text{Frob}_p$
on $e_c(M_2,{\mathbb V}_{a,b})$ has been checked in many more cases.
Consider the local systems ${\mathbb V}_{a,b}$ with $a+b\le24$ and $b\ge5$.
There are $13$ such local systems with $\dim S_{a-b,b+3}=1$; the prediction
has been checked for $11$ of them, for $p\le23$. There are also
$6$ such local systems with $\dim S_{a-b,b+3}=2$, the maximal dimension;
the prediction has been checked for all $6$ of them, for $p\le17$.

We believe the conjecture to be correct also when $a=b$ or $b=0$
after a suitable re-interpretation of several terms.
To start, we set $s_2=-1$ and define
$$S[2]=-L-1,$$ which of course is {\em not\/} equal to
${H}^1_!(M_{1,1},\mathbb{V}_0)$ but does yield the correct answer
for $$e_c(M_{1,1},\mathbb{V}_0)=L.$$
Let $S[k]=0$ for odd~$k$.
Similarly, define
$$S[0,3]=-L^3-L^2-L-1\ . $$
By analogy, let $s_{0,3}=-1$, which equals the value produced
by Tsushi\-ma's formula (the only negative value
of~$s_{j,k}$ for~$k\ge3$). We then obtain  the correct answer
for~$$e_c(A_{2},\mathbb{V}_{0,0})=L^3+L^2.$$
Finally, $S[0,10]$ is defined as $L^8+S[18]+L^9$, and more generally,
$S[0,m+1]$ includes (for $m$ odd) a contribution
$SK[0,m+1]$ defined as $S[2m]+s_{2m}(L^{m-1}+L^m)$ ---
corresponding to the Saito-Kurokawa lifts, but not a part
of ${H}^3_!(A_{2},\mathbb{V}_{m-2,m-2})$.

\subsection{Holomorphic differential forms}
We can use the Siegel modular form 
$\chi_{10}$ 
to
construct holomorphic differential forms on $M_{2,14}$.
To start, we consider the seventh fiber product $U^7$ of the universal
abelian surface 
$$U \rightarrow A_2\ .$$
We can construct $U^7$ as a quotient,
$$U^7 = \frac{\mathbb{H}_2 \times (\com^{2})^7}{\mathsf{Sp}(4,\mathbb{Z}) 
\times
(\mathbb{Z}^4)^{7}}\ .$$
We will take $z=
\left(\begin{smallmatrix}z_{11}&z_{12}\\ z_{12}&z_{22}\\ \end{smallmatrix}\right)
\in \mathbb{H}_2$ and $\zeta_i=(\zeta_{i1},\zeta_{i2})\in \com^2$ 
to be coordinates.
At the point
$$(z, \zeta_1, \ldots, \zeta_{7}) \in \mathbb{H}_2 \times (\com^{2})^7,$$
the action is
\begin{multline*}
(\left(\begin{smallmatrix}a&b\\ c&d\\ \end{smallmatrix}\right)
, (x_1,y_1), \ldots, (x_{7},y_{7})) 
\cdot (z, \zeta_1, \ldots, \zeta_{7})
= \\ \Big( 
 ({az+b})({cz+d})^{-1}, \zeta_1 (cz+d)^{-1}+x_1+y_1 
({az+b}) ({cz+d})^{-1}, 
\\ \ldots,
\zeta_{7}(cz+d)^{-1}+x_{7}+y_{7} 
({az+b})({cz+d})^{-1}
\Big) \ , 
\end{multline*}
where $x_i,y_i \in \mathbb{Z}^2$.
A direct verification using the weight 10
functional equation for  $\chi_{10}$ shows
the holomorphic 17-form
\begin{equation}\label{nnhh7}
\chi_{10}(z) \, dz_{11} \wedge dz_{12} \wedge dz_{22}
 \wedge d\zeta_{11} \wedge d\zeta_{12} \wedge
\ldots \wedge d\zeta_{71} \wedge d\zeta_{72} 
\end{equation}
on $\mathbb{H}_2 \times (\com^{2})^7$ is invariant under the
action and descends to $U^7$. 

To obtain a $17$-form on $M_{2,14}$, we consider the Abel-Jacobi
map
$M_{2,14} \rightarrow  U^7$
defined by
$$(C,p_1,\ldots,p_{14}) \mapsto  ( \text{Jac}_0(C), \omega^*_C(p_1+p_2),
\ldots, \omega^*_C(p_{13}+p_{14}))\ .$$ 
The pull-back of  \eqref{nnhh7} yields a holomorphic
$17$-form on $M_{2,14}$.
Since
 $\chi_{10}$ is a cusp form, the pull-back
extends to a nontrivial element of $H^{17,0}(\overline{M}_{2,14},\com)$.
Assuming the conjecture for local systems in genus 2,
$$H^{17,0}(\overline{M}_{2,14},\com) \cong \com^{429}  $$
with  $\Sigma_{14}$-representation
irreducible of type $[2^7]$.
We leave as an exercise for the reader to show
our construction of $17$-forms naturally 
yields the representation $[2^7]$.

The existence of $17$-forms implies the irrationality
of  $\overline{M}_{2,14}$.
Both $M_{2,n\leq 12}$ and the quotient of $M_{2,13}$
obtained by unordering the last two points are rational
by classical constructions \cite{CasFont}. Is 
$M_{2,13}$ rational?

\subsection{The compactification $\overline{M}_{2,n}$} \label{3h4}
We now assume the conjectural formula for $e_c(A_2,\mathbb{V}_{a,b})$
holds for 
$a\ge b\ge0$. As a consequence, we can compute the
Hodge numbers of the compactifications 
$\overline{M}_{2,n}$ for all $n$.

We view, as before, $M_2\subset A_2$
via the Torelli morphism. The complement is isomorphic
to $\text{Sym}^2A_1$. The local systems
$\mathbb{V}_{a,b}$
can be pulled back to $M_2$ and restricted to $\text{Sym}^2A_1$.
The cohomology of the restricted local systems can be
understood by combining the branching formula
from $\mathsf{Sp}_4$ to $\mathsf{SL}_2\times \mathsf{SL}_2$ 
with an analysis of the effect
of quotienting by the involution of $A_1\times A_1$ which
switches the factors, \cite{BvdG,Grundh,Petersen}.
Hence, the Euler characteristic
$e_c(M_2,\mathbb{V}_{a,b})$ is determined. The motives
$\wedge^2S[k]$ and~$\text{Sym}^2S[k]$ (which sum to~$S[k]^2$)
will occur.
Note $$\wedge^2S[k]\cong L^{k-1}$$ when $s_k=1$.

As observed by Getzler \cite{TRR},
the Euler characteristics $e_c(M_2,\mathbb{V}_{a,b})$ for $a+b\le N$
determine and are determined by the $\Sigma_n$-equivariant
Euler characteristics $e_c^{\Sigma_n}(M_{2,n})$ for $n\le N$.
Hence, the latter are (conjecturally) determined  for all $n$.
Next, by the work of Getzler and Kapranov \cite{GK},
the Euler characteristics $e_c^{\Sigma_n}(\M_{2,n})$ are determined
as well, since we know the Euler characteristics in genus at most 1.
Conversely, the answers for~$\M_{2,n}$ determine those
for~$M_{2,n}$.
After 
implementing the Getzler-Kapranov formalism 
(optimized for genus $2$), a calculation
of $e_c^{\Sigma_n}(\M_{2,n})$ for all $n\le22$ has been
obtained.\footnote{Thanks to Stembridge's symmetric functions 
package SF.}
Needless to say, all answers satisfy Poincar\'e duality, which
is a very non-trivial check.

The cohomology of $\overline{M}_{2,21}$ is of particular
interest to us.
The
motives~$\wedge^2S[12]$ and~$\text{Sym}^2S[12]$ 
appear for the first time for $n=21$ (with a Tate twist).
The coefficient 
of~$L\wedge^2S[12]$ in $e_c^{\Sigma_{21}}(\M_{2,21})$
equals
$$[3\,1^{18}]+[3\,2^2\,1^{14}]+[3\,2^4\,1^{10}]+[3\,2^6\,1^{6}]
+[3\,2^8\,1^{2}]+[2\,1^{19}]+[2^2\,1^{17}]$$$${}+[2^3\,1^{15}]+[2^4\,1^{13}]
+[2^5\,1^{11}]+[2^6\,1^{9}]+[2^7\,1^{7}]+[2^8\,1^{5}]+[2^9\,1^{3}]
+[2^{10}\,1^{1}]
$$
as a
$\Sigma_{21}$-representation.
As mentioned, 
\begin{equation}\label{izom}
\wedge^2S[12]\cong L^{11} . 
\end{equation}
We find thus $1939938$ independent
Hodge classes $L^{12}$,  which should be algebraic
by the Hodge conjecture. However,
$1058148$ of these classes  come with an irreducible
representation of length at least $13$. By Theorem \ref{snsn} 
obtained in 
Section \ref{rt}, the latter classes
  cannot possibly be tautological. 

The $1058148$ classes actually span a small fraction of the
full cohomology  $H^{12,12}(\overline{M}_{2,21})$.
There are also
$$124334448501272723333691$$ classes $L^{12}$ not arising     
via the isomorphism \eqref{izom} --- all
the irreducible representations of length at most $12$
arise in the coefficients.

\subsection{Genus 3}
\label{gen333}
In recent work \cite{BFG2},  Bergstr\"om, Faber, and van der Geer
have extended the point counts of moduli spaces of curves over finite
fields to the genus~$3$ case.  
For all $(n_1,n_2,n_3)$ and all prime powers $q\le17$, the frequencies 
$$\sum_C \frac{1}{\#\Aut_{\mathbb{F}_q}(C)}$$ have been computed,
where we sum over isomorphism classes of curves $C$ 
over $\mathbb{F}_{q}$ with exactly $n_i$ points
over  $\mathbb{F}_{q^i}$.
The $\Sigma_n$-equivariant  counts $\#M_{3,n}(\mathbb{F}_{q})$ are
then determined.

The isomorphism classes of Jacobians of nonsingular curves of genus~$3$
form an open subset of the moduli space $A_3$ of principally polarized
abelian varieties of 
dimension $3$.{\footnote{For $g\ge4$, Jacobians
have positive codimension in $A_g$.}} 
The complement can be easily understood.
Defining the $\mathsf{Sp}(6,\mathbb{Z})$
local systems $\mathbb{V}_{a,b,c}$ in expected manner,
the traces of $\text{Frob}_q$ on
$e_c(M_3,\mathbb{V}_{a,b,c})$ and $e_c(A_3,\mathbb{V}_{a,b,c})$ can
be computed for $q\le17$ and arbitrary $a\ge b\ge c\ge0$.

Interpretation of the results has been
quite successful for $A_3\,$: an explicit conjectural formula for
$e_c(A_3,\mathbb{V}_{a,b,c})$ has been found in \cite{BFG2}, 
compatible with all known results
(Faltings \cite{falt}, Faltings-Chai \cite{faltchai}, 
the dimension formula for the spaces of classical
Siegel modular forms \cite{Tsuy}, 
the numerical Euler characteristics \cite{BvdG}).
Two features of the formula are quite striking.
First, 
it has a simple
structure in which the endoscopic and Eisenstein contributions for genus~$2$
play an essential role.
Second, it predicts the existence of
many vector valued Siegel modular forms that are lifts, connected
to local systems of {\em regular\/} weight{\footnote{An analogous
phenomenon occurs for Siegel modular forms of genus~$2$ and level~$2$,
see \cite{BFG1}, \S6.}},
see \cite{BFG2}.

Interpreting the results for $M_3$ is much more difficult.
In fact, the stack $M_3$ cannot be
correctly viewed as an open part of $A_3$. Rather, $M_3$ is a (stacky)
double cover of the locus $J_3$ of Jacobians of nonsingular curves,
branched over the locus of hyperelliptic Jacobians.
The double covering occurs because
$$\Aut(\text{Jac}_0(C))=\Aut(C)\times\{\pm1\}$$
for $C$ a non-hyperelliptic curve,
while equality 
of the automorphism groups 
holds for $C$ hyperelliptic.
As a result, the local systems $\mathbb{V}_{a,b,c}$
of odd weight $a+b+c$
will in general have non-vanishing cohomology on $M_3$, while
their cohomology on $A_3$ vanishes. We therefore can not
expect that the main part of the cohomology of $\mathbb{V}_{a,b,c}$ on $M_3$
can be explained in terms of Siegel cusp forms when $a+b+c$ is odd.
On the other hand, the local systems of even weight provide no
difficulties. 
Bergstr\"om \cite{Bg3} has proved that
the Euler characteristics $e_c(M_3,\mathbb{V}_{a,b,c})$
are certain explicit polynomials in~$L$, for $a+b+c\le7$.

In fact, calculations \cite{BFG3} show motives not associated to
Siegel cusp forms must show up in $e_c(M_3,\mathbb{V}_{11,3,3})$
and $e_c(M_3,\mathbb{V}_{7,7,3})$, and therefore in 
$e_c^{\Sigma_{17}}(\M_{3,17})$, with the
irreducible representations of type $[3^3\,1^8]$ and $[3^3\,2^4]$.
Teichm\"uller modular forms \cite{Ichi} should play an important role  
in accounting for  the cohomology in genus 3
not explained by Siegel cusp forms.

\section{Representation theory}
\label{rt}

\subsection{Length bounds}
Consider the standard moduli filtration
\begin{equation*}
%\label{ffff}
\overline{M}_{g,n} \supset M^c_{g,n} \supset  M^{rt}_{g,n} 
\end{equation*}
discussed in Section \ref{coh}.
We consider only pairs $g$ and $n$  which satisfy the stability
condition
$$2g-2+n>0\ .$$
If $g=0$, all three spaces are equal by definition
$$\overline{M}_{0,n}=M^c_{0,n}=M^{rt}_{0,n}\ .$$
If $g=1$, the latter two are equal
$$M^c_{1,n}=M^{rt}_{1,n}\ .$$
For $g\geq 2$, all three are different.

The symmetric group $\Sigma_n$ acts on $\overline{M}_{g,n}$
by permuting the markings. Hence, $\Sigma_n$-actions are
induced on the tautological rings
\begin{equation}\label{gt559}
R^*(\overline{M}_{g,n}), \ \ R^*(M^c_{g,n}), \ \ R^{*}(M^{rt}_{g,n})\ .
\end{equation}
By Corollary \ref{fdr}, all the rings \eqref{gt559} 
are finite dimensional representations of $\Sigma_n$.

We define the {\em length} of an irreducible representation
of $\Sigma_n$ to be the number of parts
in the corresponding partition of $n$. The trivial representation
has length $1$, and the alternating representation has length $n$.
We define the length $\ell(V)$ of a finite dimensional representation $V$
of $\Sigma_n$ to be the maximum of the lengths of the 
irreducible constituents.

Our main result here bounds the lengths of the tautological 
rings in all cases \eqref{gt559}. 

\begin{thm} \label{snsn}
For the tautological rings of the moduli
spaces of curves, we have

\vspace{7pt}
\begin{enumerate}
\item[(i)]
$\qquad\ell(R^k(\M_{g,n}))\le\min\left(%n,
k+1,3g-2+n-k, \left\lfloor\tfrac{2g-1+n}{2}\right\rfloor\right)$, 
\vspace{7pt}
\item[(ii)]
$\qquad\ell(R^k(M_{g,n}^c))\le\min\left(%n,
k+1,2g-2+n-k\right)$, 
\vspace{7pt}
\item[(iii)]
$\qquad\ell(R^k(M_{g,n}^{rt}))\le\min\left(%n,
k+1,g-1+n-k\right)$.
\end{enumerate}
\end{thm}
\vspace{7pt}

The length bounds in all cases are consistent
with the conjectures of Poincar\'e duality for these 
tautological rings, see \cite{F,FPlog,HL,P}. 
For (i), the bound is invariant
under 
$$k\  \longleftrightarrow\  3g-3+n-k \ .$$
 For (ii), the bound is
invariant under 
$$k\  \longleftrightarrow\  2g-3+n-k \ $$
which is consistent with a socle in degree $2g-3+n$.
 For (iii), the bound is
invariant under 
$$k\  \longleftrightarrow\  g-2+n-k \ $$
which is consistent with a socle in degree $g-2+n$ (for $g>0$).

In genus $g$, the bound  
$\left\lfloor\tfrac{2g-1+n}{2}\right\rfloor$
in case (i) 
improves upon the trivial length 
bound $n$ only for $n\ge 2g$.

\subsection{Induction}
The proof of Theorem \ref{snsn} relies heavily on a simple
length property of induced representations of symmetric
groups.

Let $V_1$ and $V_2$ be representations
of $\Sigma_{n_1}$ and $\Sigma_{n_2}$  of lengths $\ell_1$ and $\ell_2$
respectively. 
For $n=n_1+n_2$, we view 
$$
\Sigma_{n_1}\times \Sigma_{n_2}
\subset \Sigma_n$$ in the
natural way.

\begin{prop} \label{auxi}
%Let $W$ be the representation of
%$\Sigma_{n_1}\times \Sigma_{n_2}$
%determined by $V_1\otimes V_2$. 
We have
$\ \ell\left(\rm{Ind}_{\Sigma_{n_1}\times \Sigma_{n_2}}^{\Sigma_n}
V_1 \otimes V_2
\right)=\ell_1+\ell_2\,.$
\end{prop}

\begin{proof}
To prove the result,
we may certainly assume $V_i$ is an irreducible representation
$V_{\lambda_i}$. In fact, we will exchange $V_{\lambda_i}$
for a representation more closely related to induction.

%The statement follows then immediately from the
%Littlewood-Richardson rule, but can be proved more conceptually
%as follows.
%$\alpha=(\alpha_1,\dots,\alpha_k)$ a partition of $m$,
%we have the Young subgroup
%$$ \Sigma_{\alpha_1}\times \cdots \times \Sigma_{\alpha_k} $$
%of $\Sigma_m$, and $U_{\alpha}$ is the representation of $\Sigma_m$
%obtained by inducing up the trivial representation from the
%Young subgroup.
%Recall that for $\alpha=(\alpha_1,\dots,\alpha_k)$ a partition of $m$,
%the representation $U_{\alpha}$ of $\Sigma_m$ is obtained by inducing up
%the trivial representation from the Young subgroup
%$$ \Sigma_{\alpha_1}\times \cdots \times \Sigma_{\alpha_k} $$
%of $\Sigma_m$.
For $\alpha=(\alpha_1,\dots,\alpha_k)$ a partition of $m$,
let $U_{\alpha}$ be the representation of $\Sigma_m$
induced from the trivial representation of the Young subgroup
$$ \Sigma_{\alpha}= \Sigma_{\alpha_1}\times \cdots \times \Sigma_{\alpha_k} $$
of $\Sigma_m$.
Young's rule
expresses $U_{\lambda}$  
in terms of irreducible representations $V_{\mu}$ for which $\mu$ precedes
$\lambda$ in the lexicographic ordering,
\begin{equation}\label{gvtt2}
U_{\lambda}\cong V_{\lambda} \oplus 
\bigoplus_{\mu>\lambda} K_{\mu\lambda}V_{\mu}\,.
\end{equation}
The Kostka number $K_{\mu\lambda}$ does not vanish
if and only if
$\mu$ dominates $\lambda$, which implies $\ell(\mu)\leq\ell(\lambda)$.
Therefore $V_{\lambda}-U_{\lambda}$ can be written in the
representation ring as a ${\mathbb Z}$-linear combination of $U_{\mu}$
for which $\mu>\lambda$ and $\ell(\mu)\leq\ell(\lambda)$.

If we let
$V_1 = U_{\lambda_1}$ and $V_2=U_{\lambda_2}$, the induction of
representations is
easy to calculate,
$${\rm Ind}_{\Sigma_{n_1}\times\Sigma_{n_2}}^{\Sigma_{n_1+n_2}}
U_{\lambda_1} \otimes U_{\lambda_2}
= {\rm Ind}_{\Sigma_{\lambda_1}\times\Sigma_{\lambda_2}}^{\Sigma_{n_1+n_2}}
{\bf 1}
\cong {\rm Ind}_{\Sigma_{\lambda_1+\lambda_2}}^{\Sigma_{n_1+n_2}}
{\bf 1}
= U_{\lambda_1+\lambda_2}\,,$$
where $\lambda_1+\lambda_2$ is the partition of $n$ consisting of the
parts of $\lambda_1$ and $\lambda_2$, reordered.
Hence, the Proposition is proven for $U_{\lambda_1}$ and $U_{\lambda_2}$. 
By the decomposition \eqref{gvtt2}, the Proposition follows
for  $V_{\lambda_1}$ and $V_{\lambda_2}$.
\end{proof}

\subsection{Proof of Theorem \ref{snsn}}
\subsubsection{Genus 0}
The genus 0 result
$$\ell(R^k(\M_{0,n}))\le\min\left(k+1,n-k-2\right)$$
plays an important role in the rest of the
proof of Theorem \ref{snsn} and will be proven first.

By Keel \cite{Keel}, we have isomorphisms
$$A^k(\M_{0,n})=R^k(\M_{0,n})=H^{2k}(\M_{0,n}).$$
Poincar\'e
duality then implies 
\begin{equation}\label{p459}
\ell(R^k(\M_{0,n}))
=\ell(R^{n-3-k}(\M_{0,n}))\ .
\end{equation}
The vector space $R^k(\M_{0,n})$ is generated by
strata classes of curves with $k$ nodes (and hence with $k+1$
components). The bound
$$\ell(R^k(\M_{0,n}))\le k+1$$
follows by a repeated application
of Proposition \ref{auxi} or by observing 
$$\ell(U_{\lambda})=\ell(\lambda),$$ for any partition $\lambda$
(here with at most $k+1$ parts).
The bound
$$\ell(R^k(\M_{0,n}))\le n-k-2$$
is obtained by using \eqref{p459}. \qed

\subsubsection{Rational tails}
We prove next a length bound for the subrepresentations of
$R^k(\M_{g,n})$ generated by the decorated strata classes of
curves with rational tails, in other words, by the classes
$$\xi_{A*}\Big(\prod_{v\in V(A)} \theta_v\Big)$$
as in Theorem \ref{maude}, where $A$ is a stable graph of genus
$g\ge1$ with $n$ legs and exactly one vertex of genus $g$ (and all
other vertices of genus $0$).

Let $\alpha$ be such a class of codimension $k$ and
 graph $A$.
Let $v$ be the vertex of genus $g$. We may assume  the
decoration is supported only at $v$, so 
$$\alpha=\xi_{A*}\theta_v\,.$$
The graph $A$ arises by attaching $t$ rational trees $T_i$ to $v$.
Let  $T_i$ have (before attaching) $k_i$ edges and $m_i$
legs. Let $m\ge t$ denote the valence of $v$. 
After attaching
the trees, $v$ has $m-t$ legs.
Of the $m-t$ legs, let  $s_0$  come
without $\psi$ class, $s_1$ come with $\psi^1$, 
$s_2$ come with $\psi^2$, and so forth, so 
 $$\sum_j s_j=m-t\ .$$
After discarding
the $s_j$ which vanish, we obtain a partition of $m-t$ with $p$
positive parts. The $m-t$ legs contribute a class
$\zeta_v$ of degree $\sum_j js_j$ to $\theta_v$. 
Let
$\eta_v$ be the product of the $\kappa$ classes at the
vertex and the $\psi$
classes at the $t$ legs at which the trees are attached, so 
$$\theta_v=\zeta_v\eta_v\ .$$
 Let $e$ be the degree of $\eta_v$. We
clearly have
$$ k=\sum_i k_i + t + e + \sum_j js_j \qquad {\rm and} \qquad
n=m + \sum_i m_i - 2t.$$ 

At the vertex $v$, we obtain a
representation of $\Sigma_{m-t}$ of length $p$. By the result for
genus $0$, each tree $T_i$ generates a subrepresentation $V_i$ of
$R^{k_i}(\M_{0,m_i})$ with
$$\ell(V_i)\le\min(k_i+1,m_i-k_i-2).$$
One of the $m_i$ legs is used in attaching $T_i$ to $v$.
After restricting $V_i$ to $\Sigma_{m_i-1}$, the length can only
decrease. 
Applying Proposition \ref{auxi}, we see  $\alpha$
generates a subrepresentation $V$ of $R^k(\M_{g,n})$ with
$$\ell(V)\le p+\sum_i \min(k_i+1,m_i-k_i-2).$$
%\le p+\min(\sum_i(c_i+1),\sum_i(m_i-c_i-2)).$$

Taking the first terms in the minimum expressions,
we have
\begin{equation}\label{ghrr4}
 \ell(V)\le p+\sum_i(k_i+1) = p+k-e-\sum_j js_j\,.
\end{equation}
Since
$$\sum_j js_j=\sum_{j\ge1:\,s_j>0}js_j \ge \sum_{j\ge1:\,s_j>0} 1
=p-1+\delta_{s_0,0}\,,$$ 
we see
$$p-\sum_j js_j\le
1-\delta_{s_0,0}\le 1\ .$$ 
After substituting in \eqref{ghrr4} and using $e\geq 0$, we conclude
\begin{equation}\label{gtt992}
\ell(V)\le k+1.
\end{equation}

Taking the second terms in the minimum expressions,
we have
\begin{eqnarray*}  \ell(V) &\le& p+\sum_i(m_i-k_i-2)\\ & = &
p+n-m+t+e+\sum_j js_j-k\\
&=&n-k+(p-m+t)+e+\sum_j js_j\,.\end{eqnarray*} 
The bound 
\begin{equation}\label{gby77}
e+\sum_j js_j\le g-1
\end{equation}
will be  assumed here.
Further, $p\le m-t$, so
\begin{equation}\label{gtt993}
 \ell(V)\le g-1+n-k.
\end{equation}
Combining the bounds \eqref{gtt992} and \eqref{gtt993}
yields
$$\ell(V)\le \min(k+1,g-1+n-k)$$
for any subrepresentation $V$ of $R^k(\M_{g,n})$ generated by a
class supported on a stratum of curves with rational tails
and respecting the bound \eqref{gby77}.

If the bound \eqref{gby77} is violated, then
$\theta_v$ can be expressed as a boundary class at the
vertex $v$  
by Proposition 2 of \cite{FPrel}. The boundary class will
include terms which are supported on strata of curves with
rational tails (to which the argument can be applied again)
and terms which are not supported on rational tails strata
(to which the argument can not be applied). 
Statement (iii) of Theorem \ref{snsn}  is an immediate consequence. \qed

\subsubsection{Compact type}
We prove here a length bound for the subrepresentations of
$R^k(\M_{g,n})$ generated by
the decorated strata classes of curves of
compact type with $g\geq 1$ and $n\geq 1$.

Let $\beta$ be such a class of codimension $k$ and 
graph $B$. 
Let $B$ have $b$ vertices $v_a$ of
positive genus $g_a$ and valence $n_a\ge1$. 
The graph $B$ arises
by attaching $t$ rational trees $T_i$ to the vertices $v_a$.
Let $T_i$ have (before attaching) $k_i$ edges and $m_i$ legs of which
 $l_i\ge1$
legs are used in the attachment. 
We also allow the degenerate case
in which $T_i$ represents a single node, then 
$$k_i=-1\ \ \ \text{ and } \ \ \ 
l_i=m_i=2\ .$$
At the vertex $v_a$, the attachment of the trees uses
$\widehat{n}_a$ of the $n_a$ legs. 
Just as in the case of curves with
rational tails, we keep track of the powers of $\psi$ classes
attached to the remaining $n_a-\widehat{n}_a$ legs. Let $s_{a,j}$ of those
legs come with $\psi^j$. We obtain a partition of $n_a-\widehat{n}_a$ 
with
length $p_a$. Finally, let $e_a$ be the degree of the product of
the $\kappa$ classes at $v_a$ and the $\psi$ classes at the $\widehat{n}_a$
attachment legs. We have
$$g=\sum_a g_a\,,\qquad n=\sum_i m_i+\sum_a
n_a-2\sum_i l_i\ ,$$ 
$$k=\sum_i k_i+\sum_i l_i+\sum_a e_a+\sum_{a,j} js_{a,j}\ ,$$
$$\sum_i l_i = \sum_a \widehat{n}_a = t+b-1\ .$$
The last equality follows since
 $B$ is a tree.

Applying Proposition \ref{auxi}, we see that
$\beta$ generates a subrepresentation $V$ of $R^k(\M_{g,n})$ with
$$\ell(V)\le \sum_a p_a +\sum_i \min(k_i+1,m_i-k_i-2).$$ 
Taking the first terms of the minimum expressions,
\begin{eqnarray*}
\ell(V)&\le& \sum_a p_a +\sum_i (k_i+1)\\
&=&\sum_a p_a + t+k-\sum_i l_i - \sum_a e_a -\sum_{a,j} js_{a,j}\\
&=& k+1-b+\sum_a p_a-\sum_a e_a-\sum_{a,j}
js_{a,j}\,.\end{eqnarray*} As in the rational tails case,
$$\sum_j js_{a,j}\ge p_a-1,\qquad\text{so}\qquad \sum_{a,j}
js_{a,j}\ge \sum_a p_a - b\ .$$ Therefore
$$\ell(V)\le k+1-\sum_a e_a\le k+1\ .$$ 
Taking the second terms of the minimum expressions,
\begin{eqnarray*}
&&\ell(V)\le \sum_a p_a+\sum_i (m_i-k_i-2)\\
%&=&\sum_a p_a+n-\sum_a n_a+2\sum_i l_i +\sum_i l_i+\sum_a
&&\qquad{}=\sum_a p_a+n-\sum_a n_a+3\sum_i l_i+\sum_a
e_a+\sum_{a,j}js_{a,j}-k-2t\\
&&\qquad{}= n-k+2b-2+\sum_a(p_a-n_a+\widehat{n}_a)+\sum_a e_a+\sum_{a,j}
js_{a,j}\\
&&\qquad{}\le n-k+2b-2+g-b\\ &&\qquad{}=  n-k+g+b-2\\
&&\qquad{}\le n-k+2g-2.
\end{eqnarray*} 
We have used the bound \eqref{gby77} at every vertex $v_a$. 
Therefore,
$$\ell(V)\le \min(k+1,2g-2+n-k)\le \left\lfloor\tfrac{2g-1+n}{2}
\right\rfloor $$ for any subrepresentation $V$ of $R^k(\M_{g,n})$
generated by a class supported on a stratum of curves of compact
type satisfying the bound \eqref{gby77} everywhere. As before,
Statement (ii) of Theorem \ref{snsn} is an immediate consequence. \qed
\label{oo22}

\subsubsection{Stable curves}
Finally, we prove Statement (i) by induction on the genus $g$.
Statement (i)  
has been proven already for genus $0$. 
We assume (i) is true for
 genus at most $g-1$.

We have three bounds to establish to prove (i).
The first bound to prove is
\begin{equation}\label{pkww}
\ell(R^k(\M_{g,n})) \le k+1  \ . 
\end{equation}
The bound 
holds for the subspace of $R^k(\overline{M}_{g,n})$ generated
by decorated strata classes of compact type on $\M_{g,n}$
satisfying \eqref{gby77} at all vertices by the
results of Section \ref{oo22}. If \eqref{gby77}
is violated, we use Proposition 2 of \cite{FPrel} to express
the vertex term via boundary classes and repeat.
If a class not of compact type arises,
the class must occur as a push
forward from $R^*(\M_{g-1,n+2})$. Here, we use the induction
hypothesis, and conclude the bound \eqref{pkww} for all of 
$R^k(\M_{g,n})$.
In fact, for the subspace in
$R^k(\M_{g,n})$ generated by decorated 
strata not of compact type, 
the bound $\ell\le k+1$ holds by induction.

The second bound in Statement (i) of Theorem \ref{snsn} is
\begin{equation*}
\ell(R^k(\M_{g,n})) \le 3g-2+n -k  \ . 
\end{equation*}
Since $2g-2+n-k< 3g-2+n-k$,
the bound 
holds for the subspace of $R^k(\overline{M}_{g,n})$ generated
by decorated strata classes of compact type on $\M_{g,n}$
satisfying \eqref{gby77} at all vertices by the
results of Section \ref{oo22}.
We conclude as above.
For a class not of compact
type, hence pushed forward from $R^{k-1}(\M_{g-1,n+2})$, we have
by induction
$$\ell\le 3(g-1)-2+(n+2)-(k-1)=3g-2+n-k$$
for the generated subspace.

The third and last bound to consider is 
\begin{equation*}
\ell(R^*(\overline{M}_{g,n}))\le \left\lfloor\tfrac{2g-1+n}{2}\right\rfloor\ .
\end{equation*}
The result
holds for the subspace of $R^*(\overline{M}_{g,n})$ generated
by decorated strata classes of compact type on $\M_{g,n}$
satisfying \eqref{gby77} at all vertices. 
We conclude as above using the
induction to control classes associated to
decorated strata not of compact type.
\qed

\subsection{Sharpness}
In low genera, the length bounds of Theorem \ref{snsn}
are often sharp. In fact, we have not yet seen a failure of sharpness in
genus 0 or 1. In genus $2$, the first failure occurs in
$R^2(\overline{M}_{2,3})$.
A discussion of the data is given here for
$g\leq 2$.

In genus $0$, we have $R^*(\overline{M}_{0,n}) = 
H^*(\overline{M}_{0,n},\mathbb{Q})$. Using the calculation of the
$\Sn$-representation on the latter space \cite{Getzero},
the bound
$$\ell(R^k(\M_{0,n}))\le\min\left(k+1,n-k-2\right)$$
has been verified to be sharp for $3\le n\le20$.
However, the behavior is somewhat subtle.
For $6\le n\le20$, the representation $[n-k-1,2,1^{k-1}]$
of length $k+1$ occurs for $1\le k \le \lfloor \tfrac{n-3}2 \rfloor$.
For $10\le n\le20$, the representation $[n-k,1^k]$
of length $k+1$ occurs for $0\le k \le \lfloor \tfrac{n-3}2 \rfloor$.

Consider next genus $1$. Without too much difficulty,
$R^*(\M_{1,n})$ can be shown to be  Gorenstein (with socle in degree $n$)
for $1\le n\le 5$. 
Yang \cite{Yang} has calculated the ranks of the intersection pairing
on $R^*(\M_{1,n})$ for $n\le5$ and found them to
coincide with the Betti numbers
of $\M_{1,n}$ computed by Getzler \cite{GetzSemi}. Using Getzler's
calculations of the cohomology groups as $\Sigma_n$-representations,
the length bounds for $R^k(\M_{1,n})$ are seen to be sharp
for $n\le5$.

Getzler has claimed  $R^*(\M_{1,n})$
surjects onto the even cohomology on page 973 of \cite{GetzJAMS}
(though a proof has not yet been written).
 Assuming the surjection, we can use the 
$\Sigma_n$-equivariant calculation of the Betti numbers to check
whether the length bounds for $R^k(\M_{1,n})$ are sharp for larger $n$.
We have verified the sharpness for $n\le14$.

Tavakol \cite{Tav1} has proven $R^*(M_{1,n}^c)$ is Gorenstein
with socle in degree $n-1$. Yang \cite{Yang} 
has calculated the ranks of the intersection pairing
for $n\le6$. Using the results for $R^*(\M_{1,n})$, we have verified
the length bounds for $R^k(M_{1,n}^c)$ are sharp for $n\le5$.
Probably, the $\Sigma_n$-action on $R^k(M_{1,n}^c)$ can be analyzed
more directly.

Finally, consider genus $2$. 
The length bounds are easily checked to be sharp for $n=2$.
In fact, 
$$R^*(M_{2,2}^{rt}), \ R^*(M_{2,2}^c), \ \text{and} \ R^*(\M_{2,2})$$
 all are
Gorenstein, with socles in degrees $2$,  $3$, and $5$
respectively.

The case $n=3$ is more interesting.
By Theorem \ref{snsn},
$R^2(\M_{2,3})$ and $R^4(\M_{2,3})$ have length at most $3$.
According to Getzler \cite{TRR}, 
$$H^4(\M_{2,3})\cong H^8(\M_{2,3})$$
has length $2$. Yang \cite{Yang} has shown that all the cohomology
of $\M_{2,3}$ is tautological. The Gorenstein conjecture
for $R^*(\M_{2,3})$ then implies  both 
 $R^2(\M_{2,3})$ and $R^4(\M_{2,3})$ have length $2$. By analyzing
separately
the strata of compact type, the strata for which the dual graph
has one loop, and the strata for which the dual graph has two
loops, we can indeed prove the length 2 restriction.
The length bound for $R^2(M_{2,3}^c)$ is not sharp either.

The failure of sharpness signals unexpected symmetries
among the tautological classes. Such symmetries can come
from combinatorial symmetries of the strata or from unexpected relations. 
For the failure in $R^2(\overline{M}_{2,3})$, the origin
is combinatorial symmetries in the strata. The 
 nontrivial relation of \cite{PBP} is not required.
Also, relations can exist without the failure of sharpness:
Getzler's relation \cite{GetzJAMS} in $R^2(\overline{M}_{1,4})$
causes no problems.

By Tavakol
\cite{Tav2}, $R^*(M_{2,n}^{rt})$ is
Gorenstein with socle in degree $n$. Using the Gorenstein
property, 
we have verified
the length bounds are sharp for $n\le4$ in the rational tails
case.

\section{Boundary geometry}
\label{bg}

\subsection{Diagonal classes}
We have seen the existence of non-tautological cohomology
for $\overline{M}_{1,11}$,
$$H^{11,0}(\overline{M}_{1,11}, \mathbb{C}) = \mathbb{C}, \ \ \
H^{0,11}(\overline{M}_{1,11}, \mathbb{C}) = \mathbb{C}\ .$$
As a result, the diagonal
$$\Delta_{11} \subset \overline{M}_{1,11} \times \overline{M}_{1,11}$$
has K\"unneth components which are not in $RH^*(\M_{1,11})$.
Let 
$$\iota: \M_{1,11} \times \M_{1,11} \rightarrow \M_{2,20}$$
be the gluing map.
A natural question raised in \cite{GP2} is whether
\begin{equation}\label{gh67}
\iota_*[\Delta] \notin RH^*(\M_{2,20}) \ ? 
\end{equation}
Via a detailed analysis of the action of the $\psi$
classes on $H^*(\M_{1,12},\mathbb{C})$, the following
result was proven in \cite{GP2}.

\begin{thm} \label{h889} Let 
$\Delta_{12} \subset \overline{M}_{1,12} \times \overline{M}_{1,12}$
be the diagonal. After push-forward via
the gluing map
 $$\iota: \M_{1,12} \times \M_{1,12} \rightarrow \M_{2,22}\ ,$$
we obtain a non-tautological class
$$\iota_*[\Delta] \notin RH^*(\M_{2,22})\ .$$
\end{thm}

While we are still unable to resolve \eqref{gh67},
we give a simple new proof of Theorem \ref{h889}
which has the advantage of producing new
non-tautological classes in $H^*(\M_{2,21},\mathbb{Q})$.

\subsection{Left and right diagonals}
We will study curves of genus 2 with 21 markings  
$$[C, p_1, \ldots, p_{10}, q_1, \ldots, q_{10}, r] \in \M_{2,21} \ .$$
Consider the product
$\overline{M}_{1,12} \times \overline{M}_{1,11}$
with the markings of the first factor
given by $\{p_1,\ldots, p_{10},r, \star\}$
and the markings of the second factor given by
$\{q_1,\ldots, q_{10}, \bullet\}$ .
Define the {\em left diagonal} 
$$\Delta_{L}\subset \overline{M}_{1,12} \times \overline{M}_{1,11}$$
to be the inverse image of the diagonal{\footnote{The diagonal
in $\overline{M}_{1,11} \times \M_{1,11}$ is defined by
the bijection $p_i \leftrightarrow q_i$ and $\star\leftrightarrow \bullet$.}}
$$\Delta_{11} \subset \overline{M}_{1,11} \times \overline{M}_{1,11}$$
under the map forgetting the marking $r$,
$$\pi: \overline{M}_{1,12} \times \overline{M}_{1,11}
\rightarrow \overline{M}_{1,11} \times \overline{M}_{1,11}\ .$$
The cycle $\Delta_L$ has dimension 12.

For the right diagonal, 
consider the product
$\overline{M}_{1,11} \times \overline{M}_{1,12}$
with the markings of the first factor
given by $\{p_1,\ldots, p_{10},\star\}$
and the markings of the second factor given by
$\{q_1,\ldots, q_{10}, r,\bullet\}$ .
Define 
$$\Delta_{R}\subset \overline{M}_{1,11} \times \overline{M}_{1,12}$$
to be the inverse image of the diagonal $\Delta_{11}$
under the 
map forgetting the marking $r$,
$$\pi: \overline{M}_{1,11} \times \overline{M}_{1,12}
\rightarrow \overline{M}_{1,11} \times \overline{M}_{1,11}\ ,$$
as before.

Our main result concerns the push-forwards $\iota_{L*}[\Delta_L]$ and
$\iota_{R*}[\Delta_R]$ under
the boundary gluing maps,
$$\iota_L: \overline{M}_{1,12} \times \M_{1,11} \rightarrow \M_{2,21}\ , $$
$$\iota_R: \overline{M}_{1,11} \times \M_{1,12} \rightarrow \M_{2,21}\ , $$
defined by connecting the markings $\{\star,\bullet\}$.

\begin{thm} \label{h8896} 
The push-forwards are non-tautological,
$$\iota_{L*}[\Delta_L], \ \iota_{R*}[\Delta_R] \notin RH^*(\M_{2,21})\ .$$
\end{thm}

We view the markings of $\M_{2,22}$ as given by
$\{p_1, \ldots, p_{10},r, q_1, \ldots, q_{10}, s\}$
and the diagonal $\Delta_{12}$ as defined
by the bijection
 $$p_i\leftrightarrow q_i , \ \ r \leftrightarrow s \ .$$
The cycle $\iota(\Delta_{12}) \subset \M_{2,22}$ maps birationally
to 
 $\iota_L(\Delta_{L}) \subset \M_{2,21}$ under
the map
$$\pi: \M_{2,22} \rightarrow \M_{2,21}$$
forgetting $s$. Hence,
$$\pi_* \iota_*[\Delta_{12}] = \iota_{L*}[\Delta_{L}] \ .$$
In particular, Theorem \ref{h8896} implies Theorem
\ref{h889} since tautological classes are closed under
$\pi$-push-forward.

\subsection{Proof of Theorem \ref{h8896}}
We first compute the class 
$$\iota_L^*\iota_{L*}[\M_{1,12} \times \M_{1,11}] \in 
H^*(\overline{M}_{1,12} \times \M_{1,11}, \mathbb{Q})\ .$$
The rules for such self-intersections are given in \cite{GP2}.
Since $\iota_{L}$ is an injection,
$$
\iota_L^*\iota_{L*}[\M_{1,12} \times \M_{1,11}] 
= -\psi_\star - \psi_\bullet \ .$$
As a consequence, we find
\begin{equation}
\iota_{L}^* \iota_{L*} [\Delta_L]   =  (-\psi_\star -\psi_\bullet) \cdot
[\Delta_L]\ .
\end{equation}

If $\iota_{L*}[\Delta_L] \in RH^*(\M_{2,21})$, then
the K\"unneth components of $\iota_{L}^* \iota_{L*} [\Delta_L]$
must be tautological cohomology by property (iii) of Section \ref{furth}.
Let
\begin{equation}\label{nyy}
\pi: \M_{1,12} \times \M_{1,11} \rightarrow \M_{1,11} \times \M_{1,11}
\end{equation}
be the map forgetting $r$ in the first factor.
If the K\"unneth components of $\iota_{L}^* \iota_{L*} [\Delta_L]$
are tautological in cohomology, then the
K\"unneth components of 
$$
\pi_* \iota_{L}^* \iota_{L*} [\Delta_L] \in H^*(
\M_{1,11} \times \M_{1,11} , \mathbb{Q})$$
must also be tautological in cohomology.
We compute
\begin{eqnarray*}
\pi_*\Big( ( -\psi_\star - \psi_\bullet)\cdot [\Delta_L] \Big)& = &
- \pi_*(\psi_\star \cdot [\Delta_L])
-\psi_\bullet \cdot(\pi_* [\Delta_L]) 
\\
& = & - \pi_*(\psi_\star \cdot [\Delta_L])
\end{eqnarray*}
since $\pi$ restricted to $\Delta_L$ has fiber dimension 1.

The class
$\psi_\star \in R^1(\M_{1,12})$ has a well-known boundary expression,
$$\psi_\star = \frac{1}{12}[\delta_{irr}] 
+ \sum_{S \subset \{ p_1, \ldots, p_{10}, r\}, 
\ S\neq \emptyset}  [\delta_S] \ .$$
Here, $\delta_{irr}$ is the `irreducible'
boundary divisor (parameterizing nodal rational curves with
12 markings), and $\delta_S$ is the `reducible'
boundary divisor generically parameterizing 1-nodal curves
$$ \proj^1 \cup E$$
with marking $S \cup \{ \star \}$ on $\proj^1$ and
the rest on $E$.
The intersections
$$ \delta_{irr} \times \M_{1,11} \ \cap \Delta_L, \ \ \ 
    \delta_{S} \times \M_{1,11} \ \cap \Delta_L $$
all have fiber dimension $1$ with respect to \eqref{nyy} except
when $S=\{r\}$. Hence,
$$-\pi_*(\psi_\star \cdot [\Delta_L]) = -\pi_*(
\delta_{\{r\}} \times \M_{1,11} \ \cap  \Delta_L) \ .$$
Since $\pi$ maps $\delta_{\{r\}} \times \M_{1,11} \ \cap  \Delta_L$
birationally onto $\Delta_{11} \subset \M_{1,11} \times \M_{1,11}$,
we conclude
$$
\pi_* \iota_{L}^* \iota_{L*} [\Delta_L] = - [\Delta_{11}]\ \in H^*(
\M_{1,11} \times \M_{1,11} , \mathbb{Q})\ .$$
Since $[\Delta_{11}]$ does not have a tautological K\"unneth
decomposition, the argument is complete.
The proof for $\Delta_R$ is identical.
\qed

\subsection{On $\M_{2,20}$}
We finish with a remark about the relationship between question 
\eqref{gh67} and Theorem \ref{h8896}.
Consider the map forgetting the marking labelled by $r$,
$$\pi: \overline{M}_{2,21} \rightarrow \M_{2,20} \ .$$
We easily see
$$\pi^*\left(\iota_*[\Delta_{11}]\right) = \iota_{L*}[\Delta_L] + 
\iota_{R*}[\Delta_R] \ . $$
The following push-forward relation holds by
calculating the degree of the cotangent line $\mathbb{L}_r$
on the fibers of $\pi$,
$$\pi_*\Big( \psi_r \cdot (\iota_{L*}[\Delta_{L}] + \iota_{R*}[\Delta_R])\Big) 
= 22 \cdot\iota_*[\Delta_{11}] \ .$$
As a consequence of the above two equations, we conclude the following
result.

\begin{prop} We have the equivalence:
$$\iota_*[\Delta_{11}] \in RH^*(\M_{2,20}) \ \ \ \Longleftrightarrow \ \ \
\iota_{L*}[\Delta_{L}] + \iota_{R*}[\Delta_R] \in RH^*(\M_{2,21}) \ .$$
\end{prop}
\vspace{8pt}

\subsection{Connection with representation theory}
Consider again the cycle $\Delta_L \subset \M_{1,12}\times\M_{1,11}$.
Let $$\Gamma_L\in  H^{11,11}(\M_{1,12}\times\M_{1,11}
)$$ be the K\"unneth component of $\Delta_L$. 
We can consider the $\Sigma_{21}$-submodule 
$$\mathsf{V} \subset H^{12,12}(\M_{2,21})$$
generated by $\iota_{L*}(\Gamma_L)$. The class $\Gamma_R$
can be defined in the same manner, and  
$\iota_{R*}(\Gamma_R) \in \mathsf{V}$.

The class $\Gamma_L$ is alternating for the 
symmetric group $\Sigma_{10}$ permuting the points $p_1, \ldots, p_{10}$.
Similarly, $\Gamma_L$ is alternating for the 
$\Sigma_{10}$ permuting the points $q_1, \ldots, q_{10}$.
Let 
$$\Sigma_{10}\times \Sigma_{10}\subset \Sigma_{21}$$ be the
associated subgroup. 
We can then consider the $\Sigma_{21}$-module defined by
$$\widetilde{\mathsf{V}} = 
\text{Ind}^{\Sigma_{21}}_{\Sigma_{10}\times \Sigma_{10}} \big( 
\alpha\otimes\alpha\big),$$
where $\alpha$ is the alternating representation.
In fact, the representation 
$\widetilde{\mathsf{V}}$ decomposes as
$$[1^{21}] + \sum_{i=0}^9\, [3\,2^i\,1^{18-2i}]
+2\sum_{j=1}^{10}\,[2^j\, 1^{21-2j}].
$$
We can write the coefficient of 
$L\wedge^2S[12]$ in $H^{12,12}(\M_{2,21})$
discussed in
Section \ref{3h4} as
$$  \sum_{i=0,\,\text{even}}^9\, [3\,2^i\,1^{18-2i}]
+\sum_{j=1}^{10}\,[2^j\, 1^{21-2j}].
$$
We conjecture the 
canonical map
$$ \widetilde{\mathsf{V}} \rightarrow {\mathsf{V}}$$
is simply a projection onto the
above subspace of 
$H^{12,12}(\M_{2,21})$.

\vspace{+20 pt}

\noindent Institutionen f\"or Matematik \hfill Department of Mathematics\\
\noindent Kungliga Tekniska H\"ogskolan \hfill Princeton University\\
\noindent 100 44 Stockholm, Sweden \hfill Princeton, NJ 08544\\
\noindent faber@math.kth.se \hfill rahulp@math.princeton.edu

\end{document}